\documentclass[a4paper, 11pt]{article}
\usepackage{enumerate}
\usepackage{latexsym}
\usepackage{amssymb,amsmath,amsopn}
\usepackage[dvips]{graphicx}   %To insert ps figures
\usepackage{color,epsfig}      %To insert ps figures
\usepackage{graphicx}
\graphicspath{{./abrak/}}
\usepackage{amsthm}
\usepackage{cancel}
\usepackage{color}
\usepackage{authblk}
\usepackage{comment}

\usepackage{hyperref}
\hypersetup{
	colorlinks=true,
	linkcolor=black,
	anchorcolor=black,
	citecolor=black,
	filecolor=black,
	menucolor=red,
	runcolor=black,
	urlcolor=black,
}

               {\begin{list}{}{\leftmargin#1\rightmargin#2}\item{}}%
               {\end{list}}

\newtheorem{theorem}{Theorem}
\newtheorem{lemma}{Lemma}
\newtheorem{prop}{Proposition}

\newtheorem{definition}{Definition}

\newcommand{\R}{\mathbb R}

\definecolor{green}{rgb}{0.0, 0.75, 0.25} %\fk
\def\eea{\end{eqnarray}}
\renewcommand{\emph}{\textit}
\newcommand{\bl}{\boldsymbol{\lambda}}
\newcommand{\bp}{\boldsymbol{p}}
\newcommand{\up}{\underline{p}}
\newtheorem{remark}{Remark}

\DeclareMathOperator{\proj}{proj}

\begin{document}

\title{On an abrasion motivated fractal model}

\author[1]{Bal\'azs B\'ar\'any\thanks{BB acknowledges support from grant NKFI~FK134251 and grant NKFI KKP144059 "Fractal geometry and applications" Research Group.}}
\author[2,3]{G\'abor Domokos\thanks{GD acknowledges support from the NKFIH Hungarian Research Fund grant 134199 and grant BME FIKP-V\'IZ}}
\author[4]{\'Agoston Szesztay\thanks{\'ASz is supported by the Doctoral Excellence Fellowship Programme (DCEP) is funded by the National Research Development and Innovation Fund of the Ministry of Culture and Innovation and the Budapest University of Technology and Economics under a grant agreement with the National Research, Development and Innovation Office.}}
\affil[1]{Department of Stochastics, Institute of Mathematics, Budapest University of Technology and Economics, Műegyetem rpk. 1-3., Budapest, Hungary, H-1111}
\affil[2]{Department of Morphology and Geometric Modeling, Budapest University of Technology and Economics, Műegyetem rpk. 1-3., Budapest, Hungary,	H-1111}
\affil[3]{HUN-REN--BME Morphodynamics Research Group, Budapest University of Technology and Economics, Műegyetem rpk. 1-3., Budapest, Hungary, H-1111}
\affil[4]{Department of Mechanics, Materials and Structures, Budapest University of Technology and Economics, Műegyetem rpk. 1-3., Budapest, Hungary,	H-1111}

	\maketitle
	
	\begin{abstract}
		In this paper, we consider a fractal model motivated by the abrasion of convex polyhedra, where the abrasion is realised by chipping small neighbourhoods of vertices. {After providing a formal description of the successive chippings, we show that the net of edges converge to a compact limit set under mild assumptions. Furthermore, we study the upper box-counting dimension and the Hausdorff dimension of the limiting object of the net of edges after infinitely many chipping.}%We study the upper box-counting dimension of the limiting object after infinitely many chipping.
	\end{abstract}

%\date{\textrm{\today}}
	
	\section{Introduction}

%\begin{abstract}
%In a recent paper, cubic averages (i.e., six faces, eight vertices, and 12 edges) were found for global,
%primary fragmentation patterns modeled by primitive mosaics, which are generated by collective bisection of convex polyhedra.
%Numerical evidence was presented that local, individual bisection on convex polyhedra results in similar averages. Here, we %present an algorithm for local bisection that produces exact cubic averages. We perform numerical tests. The essence of the algorithm is that it restricts fragmentation to local truncations.
%\end{abstract}

%\tableofcontents

%\section{Global truncations of convex polyhedra}

\subsection{Motivation}
Abrasion under collisions (also called \emph{collisional abrasion} or \emph{chipping}) is one of the main geological processes governing the evolution of natural shapes, ranging from pebbles to asteroids \cite{domokos_2014, domokos_2009, szabo_universal_2018}. The process is driven by a sequence of discrete collisions where  \emph{abraded particle} collides with \emph{abraders}. Based on their energy,
collisions emerge in three well-separated phases \cite{pal_2021}: large energy collisions belong to the \emph{fragmentation phase}, where cracks propagate through the entire particle, which is ultimately split into several parts of comparable volume. Medium energy collisions belong to the \emph{cleavage phase}, where the removed volume is smaller, but the crack propagates into the interior of the particle. Finally, in the 
\emph{abrasion phase} (also called the chipping phase), we consider small energy collisions where cracks remain in the vicinity of the surface and a small portion of the material is removed.

Geometric models of the high energy (fragmentation) phase and of the low energy (abrasion)  phase have been studied in considerable detail. In the case of fragmentation, geometric models consider the bisection of convex polyhedra by random planes and study the combinatorial and metric properties of the descendant polyhedra \cite{barany_2024, Domokos18178, domokos_universality_2015}. In the case of abrasion, 
considering the limit where collision energy approaches zero led to the study of mean field geometric partial differential equations (PDEs) \cite{Bloore1977, Domokos_Gibbons_2012, Firey1974, Hamilton_Wornstones}, 
describing the evolution of shape as a function of continuous time. If one considers the original collision process associated with finite impact energies, then discrete time evolution models appear to be a natural choice \cite{domokos_2014, pal_2021}. While no rigorous result is known that connects discrete-time models to PDE models, their predictions show a very good qualitative match \cite{domokos_2014, pal_2021}, suggesting that the geometric study of discrete-time collision models may shed light on general features of shape evolution.

While the discreteness of shape evolution models referred so far only to time, in such models, convex polyhedra are the natural choice as geometric approximations of the studied particle. This choice is natural not only because (as we outline below) discrete time steps are best understood
on discrete geometric objects but also because the 3D scanned images of particles on which computer
codes can operate are also polyhedral objects \cite{Ludmany_Domokos_(2018)}.

The low energy, abrasion phase, geometric models of collisions are truncations of the polyhedral model, which remove small portions of its volume. If the latter is sufficiently small, then, from the combinatorial point of view, 
we can distinguish between three kinds of local events where (a) one vertex is removed and one face is created, (b) one edge is removed and one face is created and (c) one face retreats parallel to itself and the combinatorial structure of the polyhedron remains invariant. These three local events do not differ from the point of view of collision energy. However, they differ from the point of view of the relative size and shape of the abraded particle concerning the abrading particle \cite{Bloore1977, Domokos2014, Domokos_Sipos_Varkonyi_2009}. In particular, event (a) corresponds to the case when the abrader is much larger, and event (c) to the case when the abrader is much smaller than the abraded particle. While none of these three events can, on its own, fully capture collisional shape evolution in the low energy (abrasion) phase. Still, the individual study of these events can provide both geometric and physical insight. Moreover, in some cases, one single event reproduces
global shape features with remarkable accuracy.

Our goal is the detailed geometric description of the event (a) when one vertex is removed in each step of the shape evolution process. Such discrete steps are called \emph{chipping events} \cite{Novak-Szaboeaao4946, Redner_Krapivsky_2007}. In our paper, we will remove \emph{all} vertices simultaneously, and we refer to this collective event as a single chipping event. The planar version of the chipping event was studied earlier in \cite{Redner_Krapivsky_2007}, revealing the emergence of fractal-like contours. Our goal is to
offer a full and rigorous geometric study of this phenomenon in three dimensions. As noted above, event (a) corresponds to the case where the abrading object is much larger than the abraded particle, and this is a realistic approximation of pebbles carried in mountain rivers and evolving under collisions with the riverbed. Figure \ref{fig:example} shows an andesite rock that has been abraded in the Poprad River in the Tatra mountains. As a visual comparison, we show a polyhedron with 2912 faces, which was produced from a cube via six consecutive chipping events. Figure \ref{fig:chop0} shows the vicinity of one vertex of the cube as well as the Apollonian gasket for visual comparison.

\begin{figure}[h]
	\centering
	\includegraphics[width=10cm]{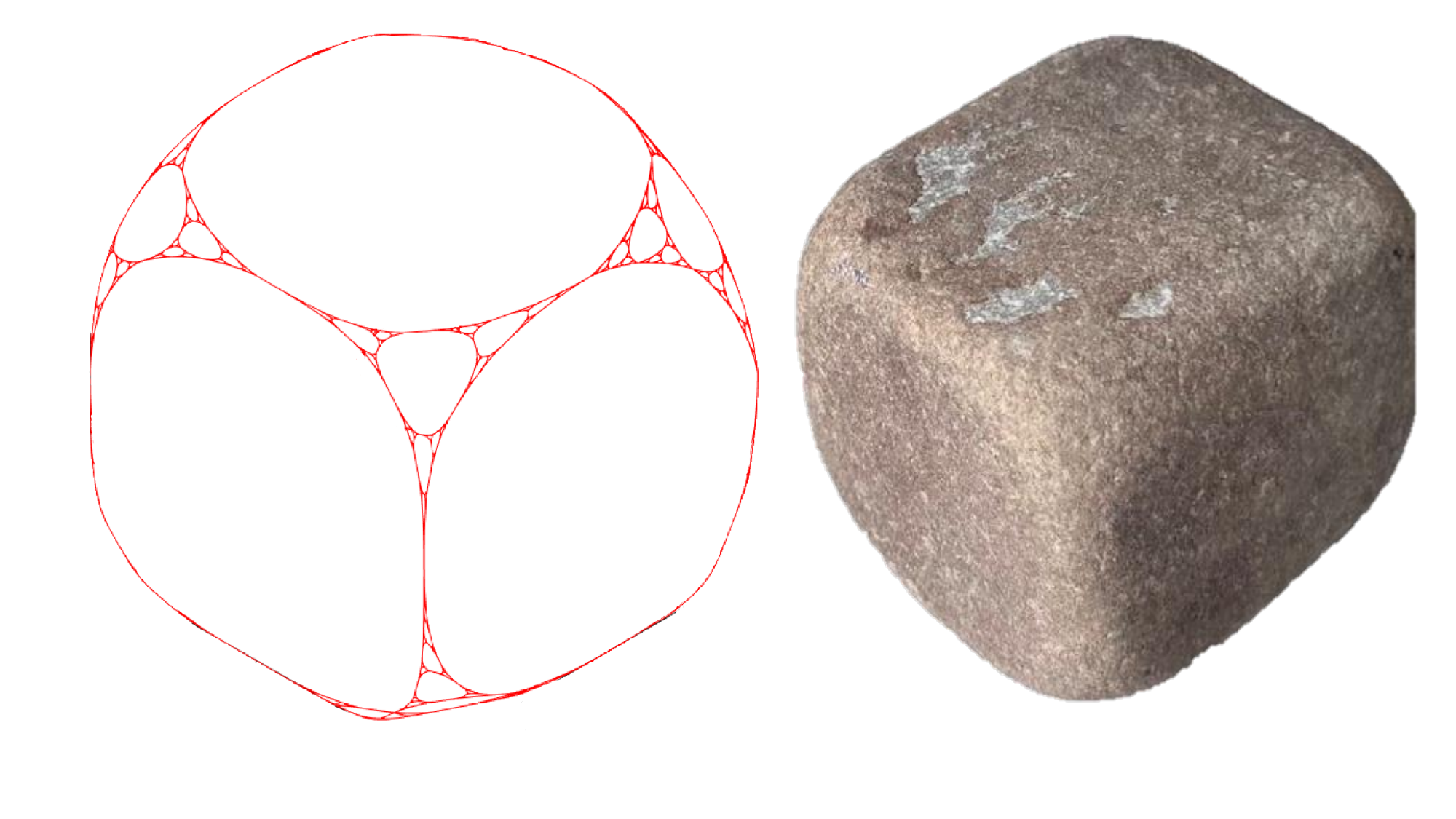} 
	\caption{Left: Cube after six iterative chipping events with random orientation, having 2912 faces. Right: Moderately abraded andesite rock recovered from the Poprad River at the foot of the High Tatra mountains.}\label{fig:example}
\end{figure}

    \begin{figure}[ht]
		\centering
		\includegraphics[height=5cm]{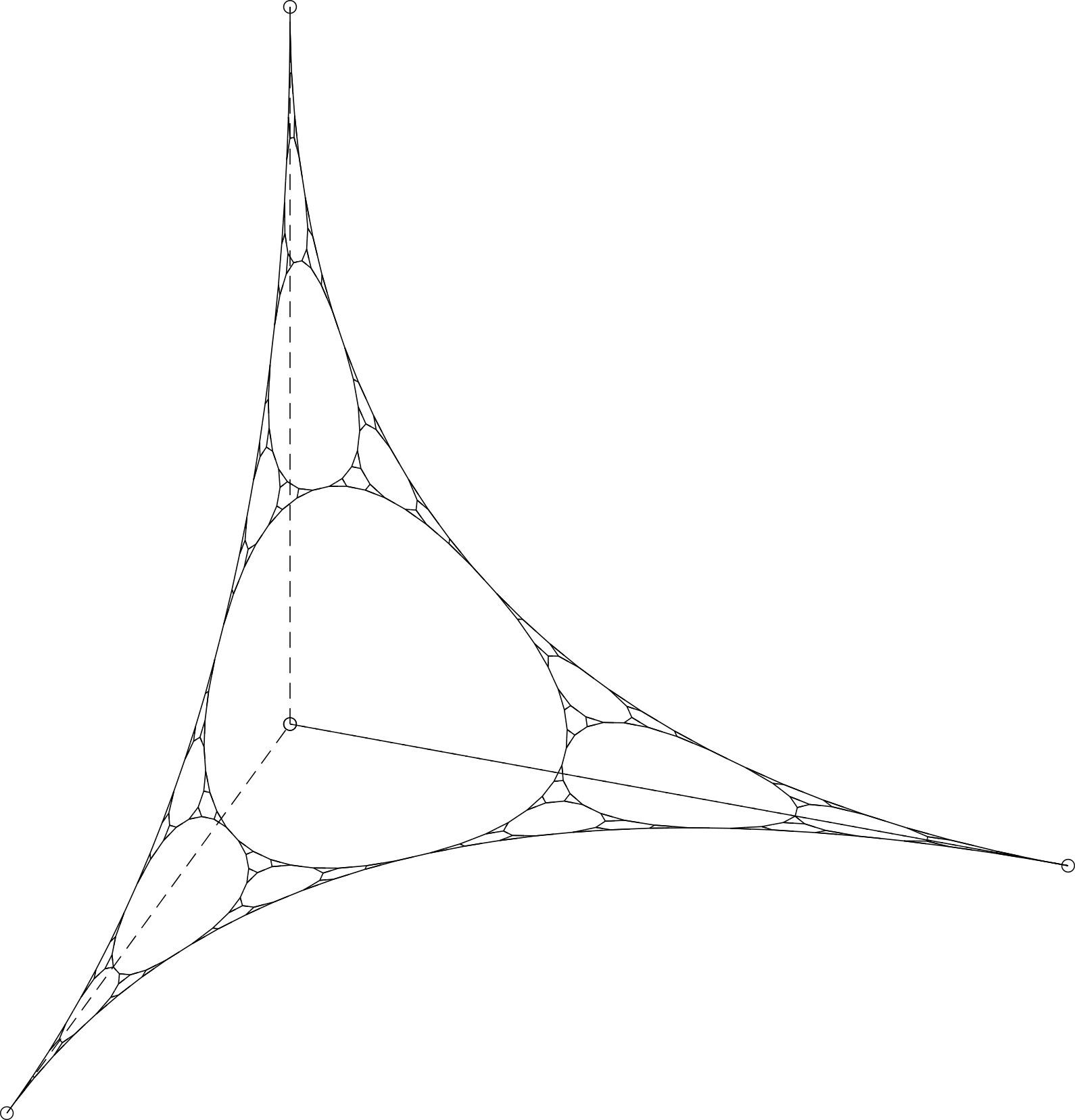}
            \includegraphics[height=5cm]{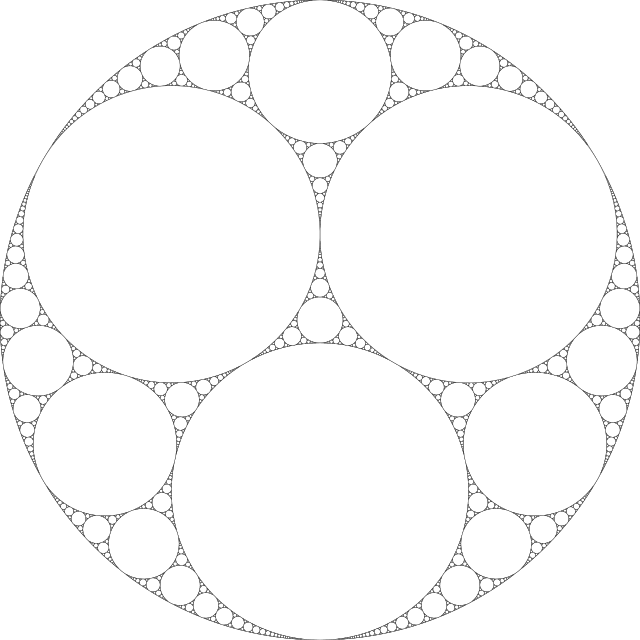}
            \caption{left: Figure of a simulation of chipping in the proximity of an original vertex of the initial polyhedron. right: the Apollonian gasket %figure from wikipedia
            }\label{fig:chop0}
	\end{figure}

 Motivated by this visual analogy, we are interested in the geometric description of the limit where the number of chipping events approaches infinity. In this limit, the polyhedron (more precisely, its edge network) approaches a fractal-like object, and our {first} main result (Theorem \ref{thm:main}) determines the box-counting dimension of this object. {Based on the recent results of \cite{BHR,MS}, the Hausdorff dimension of that object in a less generic case is also calculated (Theorem~\ref{thm:main2}).}
	
 In fractal geometry, one of the cardinal questions is the dimension of the object under consideration. There are several different kinds of dimensions that are devoted to measuring how much the fractal set is spread. An advantage of the box-counting dimension is that there are available methods that allow us to study the dimension of actual 3D scanned images of particles, but unfortunately, these methods might give some relatively good approaches only at certain scales.

 It turns out that our model is strongly connected to the self-affine sets, which have been extensively studied in the last decades; see \cite{BHR, Fal1, Fal2, HR, JPS, R}. The dimension theory of such objects is highly non-trivial. For instance, even in cases where a formula for the value of the dimension is known, it cannot be calculated explicitly, only implicitly, and it can be approximated well only in some cases; see \cite{Morris22, PoVy15}. This is due to the extremely difficult structure of the group of matrix products.
 
 In our case, this difficulty arises as well. Namely, we can give only an implicit formula for the box-counting dimension, which depends only on the chippings. That is, the value of the dimension is independent of the initial polyhedron.

The structure of the paper is as follows: In Section~\ref{sec:basic}, we give a definition for chipping, and for further analysis, we introduce the local chart representation of simple convex polyhedra, and we define a sequence of iterated function systems (IFS) corresponding to the chipping. Finally, section~\ref{sec:proof} is devoted to the proof of Theorem \ref{thm:main}, {and section~\ref{sec:proof2} is devoted to proof of Theorem~\ref{thm:main2}}.

%Collisional abrasion is one of the main drivers of shape evolution processes in geomorphology \cite{jerolmack, szabo_universal_2018}

	\section{The model, the iterated function scheme representation, and the dimension}\label{sec:basic}
	
	\subsection{The chipping model and the limit set}
	
	Let $P$ be a convex polyhedron, that is, let $\mathcal{V}$ be a finite subset of $\R^3$ and let $P=\mathrm{conv}(\mathcal{V})$ be the convex hull of $\mathcal{V}$ such that every point of $\mathcal{V}$ is an extremal point of $P$. We call $\mathcal{V}=\mathcal{V}(P)$ the set of {\it vertices} of $P$. Furthermore, let $\mathcal{E}=\mathcal{E}(P)$ be the set of {\it edges}. That is, for any two distinct $A, B\in\mathcal{V}$, $$[A, B]=\{tA+(1-t)B:0\leq t\leq1\}\in\mathcal{E}(P)$$ if and only if every $x\in[A, B]$ has a unique representation by the convex combination of vertices using only $A$ and $B$. Let $H$ be the net of edges, i.e. $H=H(P)=\bigcup_{[A,B]\in\mathcal{E}(P)}[A,B]$.
	
	Let us index the vertices of $P$ by a finite set $I=I(P)$, i.e. $\mathcal{V}(P)=\{A_i:i\in I(P)\}$. {For simpler notations, in some cases, we refer directly to members of $\mathcal{V}(P)$ by their indices. } Let $E(P)\subseteq I(P)\times I(P)$ be such that $(i,j)\in E(P)$ if and only if $[A_i,A_j]\in\mathcal{E}(P)$. We use the convention that $E(P)$ is symmetric, i.e., $(i,j)\in E(P)$ if and only if $(j,i)\in E(P)$. For a $j\in I$, let $N(j)$ be the set vertices, which are the neighbours of $j$, that is, $i\in N(j)$ if and only if $(j,i)\in E(I)$. We call the convex polyhedron $P$ {\it simple} if $\#N(j)=3$ for every $j\in I(P)$. Furthermore, let $\mathcal{F}(P)$ be the set of faces of the polyhedron $P$, and for an $S\in\mathcal{F}(P)$, denote by $\mathbf{n}(S)$ a unit normal vector perpendicular to $S$. Finally, for a $j\in I$, let $M(j)$ be the set of faces $S$ such that $A_j$ is a vertex of $S$.
	
	Now, we will define the way how a convex polyhedron evolves under the chipping algorithm. Let $j\in I(P)$ and let $\mathbf{n}_j$ be a vector such that $A_j+\varepsilon\mathbf{n}_j$ is an interior point of $P$ for every sufficiently small $\varepsilon>0$, and $\mathbf{n}_j$ is not parallel with any $\mathbf{n}(S)$ for every $S\in M(j)$. Let $\Psi_j$ be the hyperplane with normal vector $\mathbf{n}_j$, going through the point $A_j+\varepsilon_j\mathbf{n}_j$, and let $\Phi_j$ be the closed half-space determined by $\Psi_j$ such that $A_j$ is an interior point of $\Phi_j$, where $\varepsilon_j$ are chosen such that $(P \cap \Phi_j) \cap (P \cap \Phi_k) = \emptyset$ for every $j\neq k\in I(P)$. By chipping, we mean the removal of such $P \cap \Phi_j$ pyramids from all vertices of $P$ and the new chipped polyhedron is $\overline{\bigcup_{i\in I(P)}P\setminus\Phi_i}$. \\By simple geometric arguments, it is easy to see that the chipping of vertices generates a simple polyhedron. Thus, from now on, we will always assume without loss of generality that $P$ is simple. Furthermore, for the chipping of simple polyhedra, we can give the following simpler definition: 
	
	\begin{definition}[chipping]\label{def:chipping} Let $P$ be a simple convex polyhedron with vertices $\mathcal{V}(P)$ indexed by $I$ and edges $\mathcal{E}(P)$ indexed by $E(P)\subseteq I(P)\times I(P)$. Let $\underline{p}=(p_{ij})_{(i,j)\in E(P)}$ be a vector of positive reals such that for every $i,j\in I(P)$ with $(i,j)\in E(I)$, $p_{ij}\in(0,1)$ and $p_{ij}+p_{ji}<1$. Let us call the vector $\underline{p}$ the {\normalfont chipping rate} vector. We define the chipped polyhedron $\mathcal{C}_{\underline{p}}(P)$ as follows: let the set of vertices
		$$
		\mathcal{V}(\mathcal{C}_{\underline{p}}(P)):=\{A_j+p_{ji}\overrightarrow{A_jA_i}:i,j\in I(P)\text{ such that }(i,j)\in E(P)\},
		$$
		and $\mathcal{C}_{\underline{p}}(P)=\mathrm{conv}(\mathcal{V}(\mathcal{C}_{\underline{p}}(P)))$.
		
		Let us index the vertices of $\mathcal{C}_{\underline{p}}(P)$ by $E(P)$, that is, let $I(\mathcal{C}_{\underline{p}}(P))=E(P)$ and $A_{ji}:=A_j+p_{ji}\overrightarrow{A_jA_i}$. 
	\end{definition}
	
	 \begin{comment}
	    It is easy to see that for any convex polyhedron $P$ and for any chipping rate vector $\underline{p}$, the polyhedron $\mathcal{C}_{\underline{p}}(P)$ is simple. 
	\end{comment}

	%Observe that $E(P)=I(\mathcal{C}_{\underline{p}}(P))$. Hence, observe the that the vector $\underline{P}$ can be both indexed by $E(P)$ and $I(\mathcal{C}_{\underline{p}}(P))$.

	During the shape evolution of a simple polyhedron $P$ when one of its vertices $A_j$, which was in connection with $A_i, A_k$ and $A_l$, is chipped, it is replaced by a new face composed of three new vertices, which are named after the vertices they are created between $A_{ji}, A_{jl}$ and $A_{jk}$. {In these new vertices, the index of the chipped vertex, which is replaced by a new face, is noted first.} Chipping planes can only intersect edges starting from the chipped vertex. The new vertices are placed on the edges according to proportion $p_{ji}, p_{jl}, \text{ and } p_{jk}$. Similarly, chipping $A_{ji}$ creates the new vertices $A_{jiij}, A_{jijk}, A_{jijl}$ dividing the edges according to $p_{jiij}, p_{jijl}$, and  $p_{jijk}$, etc. Moreover, it is easy to see that any $A_{ji},A_{kl}\in\mathcal{V}(\mathcal{C}_{\underline{p}}(P))$, $[A_{ji},A_{kl}]\in\mathcal{E}(\mathcal{C}_{\underline{p}}(P))$ if and only if $j=k$ or $j=l$ and $i=k$.

	%\begin{definition}
	%	Let $j\in I(P)$ and let $i,k,l\in N(j)$ and let $\sigma$ be a permutation of $N(j)$. We say that the local chart $F_{j,\sigma}$ is adapted to the chipping if for every ordered pair $ij$ of vertices  $i,j\in I$ such that $(i,j)\in E(P)$, we have $p_{ij}<\lambda_{ij}$.
	%\end{definition}
	
	Let $j\in I(\mathcal{C}_{\underline{p}}(P))$. There exist a unique $j'\in I(P)$ and a unique $i'\in N(j')$ such that $j=j'i'$. We call $j'$ the {\it mother} of $j$, and we call $i'$ the {\it father} of $j$. Furthermore, If $P$ is simple, then $k\in N(j)$ if and only if $k=i'j'$ or there exists $i'\neq k'\in N(j')$ such that $k=j'k'$. We call the vertex $i'j'\in N(j'i')$ the {\it sibling} of $j=j'i'$. Further, we call the vertices $j'k',j'l'\in N(j'i')$, where $k',l'\in N(j')$ for $k'\neq i'$ and $l'\neq i'$, the {\it cousins} of $j'i'$. Let us denote the index of the sibling of $j$ by $s(j)$.  {For a visual representation of the chipping in a neighbourhood of a vertex and for the family relations, see Figure~\ref{fig:chop}.}
	
	\begin{figure}[h]
		\centering
		\includegraphics[width=10.5cm]{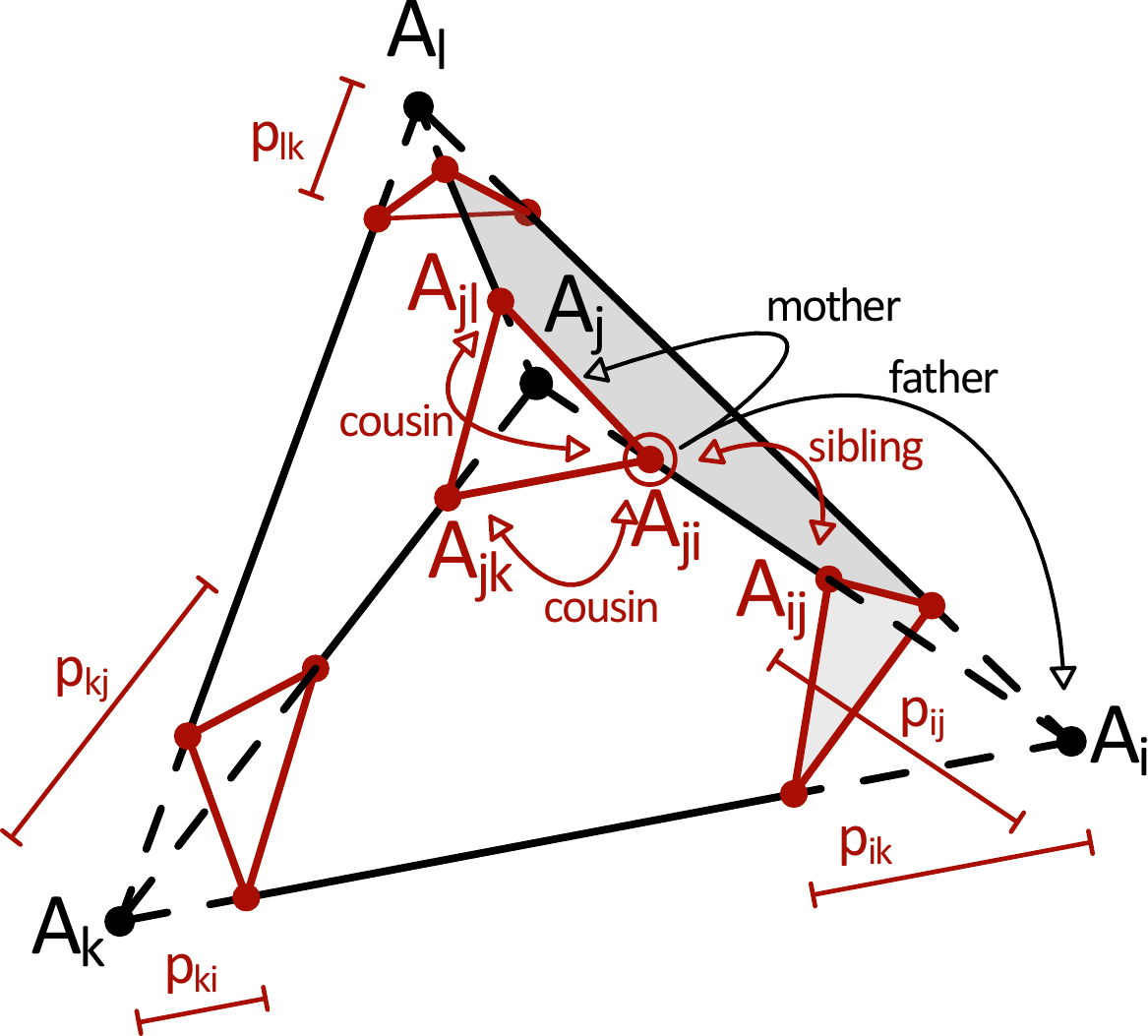}
            \caption{Chipping of a tetrahedron. Black colour refers to the initial tetrahedron $P_0$, and red refers to the effect of chipping and chipping rates. {Arrows indicate the family relationships of the $A_{ji}$.}}\label{fig:chop}
	\end{figure}
 Note that the advantage of the described indexation is that there is a one-to-one correspondence between the indexes $E(P)$ of edges of $P$ and the vertices $I(\mathcal{C}_{\underline{p}}(P))$ of $\mathcal{C}_{\underline{p}}(P)$ since $E(P)$ is symmetric. In particular, $I(\mathcal{C}_{\underline{p}}(P))=E(P)$.
	
	Using the chipping algorithm of Definition~\ref{def:chipping}, we can define a sequence of simple convex polyhedra $P_n$ as follows: For any initial simple, convex  polyhedra $P_0=P$, let $\underline{p}_0$ be a chipping rate as in Definition~\ref{def:chipping}, and let $P_1=\mathcal{C}_{\underline{p}_0}(P_0)$. Suppose that the simple convex polyhedron $P_n$ is defined, then let $\underline{p}_n$ be a chipping rate on $E(P_n)$ and let $P_{n+1}=\mathcal{C}_{\underline{p}_n}(P_n)$. We call the sequence $P_n$ of polyhedra a {\it chipping sequence}.
	
	\begin{remark}
		Note that as $n$ increases by one, the words describing the elements of $I(P_{n})$ are going to be $2^{n}$ long combinations of the indices of $I(P_0)$. Words describing elements of $E(P_{n})$ are twice as long words.
	\end{remark}
	
	The main object of our study is the limit set of the net of edges $H(P_n)$ of the chipping sequence $P_n$. {As we will see, there exists a unique compact set $X$ to which the sequence $H(P_n)$ is converging in some proper sense (Hausdorff metric), and this set shows fractal-like properties strongly related to self-affine sets. For a discussion of this phenomenon, see Section~\ref{sec:boxdim}.}
 
 Let us now define the Hausdorff metric of compact subsets of $\R^3$. For a set $X$ and $\delta>0$, let $$[X]_\delta=\{y\in\R^3:\text{ there exists }x\in X\text{ such that }\|x-y\|<\delta\},$$ where $\|.\|$ denotes the usual Euclidean norm. We define the Hausdorff metric between two compact sets $X, Y\subset\R^3$ 
	$$
	d_H(X,Y)=\inf\{\delta>0:Y\subset[X]_\delta\text{ and }X\subset[Y]_\delta\}.
	$$
	It is well known that the set of compact subsets of $\R^3$ endowed with $d_H$ forms a complete metric space, see \cite[Theorem~3.16]{Fal}. It is easy to see that $d_H(\bigcup_{i=1}^nX_i,\bigcup_{i=1}^nY_i)\leq\max_{i=1,\ldots,n}d_H(X_i,Y_i)$ and for any Lipschitz map $f$ with Lipschitz constant $C>0$, $d_H(f(X),f(Y))\leq Cd_H(X,Y)$. 
	
	\begin{prop}\label{prop:limitset}
		Let $P_n$ be a chipping sequence such that there exists a $\delta>0$ such that for every $n\geq1$ and $i\in I(P_n)$, $\delta<p_{i}$ and for every $j\in N(i)$, $p_{i}+p_{j}<1-\delta$. Then, there exists a unique compact set $X$ such that $d_H(H(P_n), X)\to0$ as $n\to\infty$, where $H(P)$ is the net of edges of $P$.
	\end{prop}

	We will give the proof of Proposition~\ref{prop:limitset} later at the end of Section~\ref{sec:existence}. In Figure~\ref{fig:example}, one can see the comparison of some realisations of the limit set of the chipping sequence and an abraded andesite rock found in Poprad River.

	For short, we say that the chipping rates $\bp=(\underline{p})_{n=1}^\infty$ are {\it regular} if there exists a $\delta>0$ such that for every $n\geq1$ and $i\in I(P_n)$, $\delta<p_{i}$ and for every $j\in N(i)$, $p_{i}+p_{j}<1-\delta$. The regularity of the chipping sequence roughly means that there is some fixed percentage, such that at least that percentage of every edge is chipped in a neighbourhood of a vertex, but a fixed percentage of every edge is kept.

	\subsection{Local representation of simple convex polyhedra} 
	
	Let $P$ be a simple convex polyhedron. Now, we define a local representation of the edge net $H(P)$ of $P$ by affine mappings. Let us denote the usual orthogonal basis of $\mathbb{R}^3$ by $\{e_1,e_2,e_3\}$. Let $L=\bigcup_{i=1}^3[0,e_i]$.
	
	\begin{definition}[Local chart]
		Let $P$ be an arbitrary, simple, convex polyhedron. Let $\bl=(\lambda_{ij})_{(i,j)\in E(P)}$ be a vector of positive reals such that for every $i,j\in I(P)$ with $(i,j)\in E(I)$, $\lambda_{ij}\in(0,1)$ and $\lambda_{ij}+\lambda_{ji}=1$.
		
		Furthermore, for every $j\in I(P)$, let $\sigma_j\colon\{1,2,3\}\to N(j)$ be a permutation of the {neighbouring vertices } of $j$ and let $\sigma=(\sigma_j)_{j\in I(P)}$. Let us define the matrix \begin{equation}\label{eq:Lambda}\Lambda_{j,\sigma,\bl}. =\begin{bmatrix}
			\lambda_{j\sigma_j(1)} & 0 & 0 \\ 0 & \lambda_{j\sigma_j(2)} & 0 \\ 0 & 0 & \lambda_{j\sigma_j(3)}
		\end{bmatrix}.%= \begin{bmatrix}
	%		\lambda_{ji} & 0 & 0 \\ 0 & \lambda_{jk} & 0 \\ 0 & 0 & \lambda_{jl}
	%	\end{bmatrix}.
\end{equation}
		Let  $F_{j,\sigma,\bl}\colon\R^3\mapsto\R^3$ be such that
		$$
		F_{j,\sigma,\bl}(x)=\begin{bmatrix} \overrightarrow{A_jA_{\sigma_j(1)}} & \overrightarrow{A_jA_{\sigma_j(2)}} & \overrightarrow{A_jA_{\sigma_j(3)}}\end{bmatrix}\Lambda_{j,\sigma,\bl} x+A_j,
		$$
		where $\overrightarrow{AB}$ denotes the vector with initial $A$ and endpoint $B$. We call $F_{j,\sigma,\bl}$ the {\normalfont local chart map} of $j$ with permutation $\sigma$ and rate $\bl$, and we call $F_{j,\sigma,\bl}(L)$ the {\normalfont local {neighbourhood}} of $j\in I(P)$.
	\end{definition}

For short, let $\boldsymbol{A}_{j,\sigma}:=\begin{bmatrix}\overrightarrow{A_{j}A_{\sigma_{j}(1)}} & \overrightarrow{A_{j}A_{\sigma_{j}(2)}} & \overrightarrow{A_{j}A_{\sigma_{j}(3)}}\end{bmatrix}$. By definition, $F_{j,\sigma,\bl}(0)=A_j$ and $F_{j,\sigma,\bl}(\underline{e}_\ell)=\lambda_{j\sigma(\ell)}\cdot\overrightarrow{A_jA_{\sigma(\ell)}}+A_j$ for any $\ell=1,2,3$. Hence, $H(P)=\bigcup_{j\in I(P)}F_{j,\sigma,\bl}(L)$. Furthermore, we call the set of functions $\mathcal{F}_{\sigma,\bl}=\{F_{j,\sigma,\bl}\}_{j\in I(P)}$ as the {\it chart} of $H(P)$ with respect to the permutations $\sigma$ and rate $\bl$. For a visual representation of the local charts, see Figure~\ref{fig:localchart}.

	\begin{figure}[h]
	\centering
	\includegraphics[width=13cm]{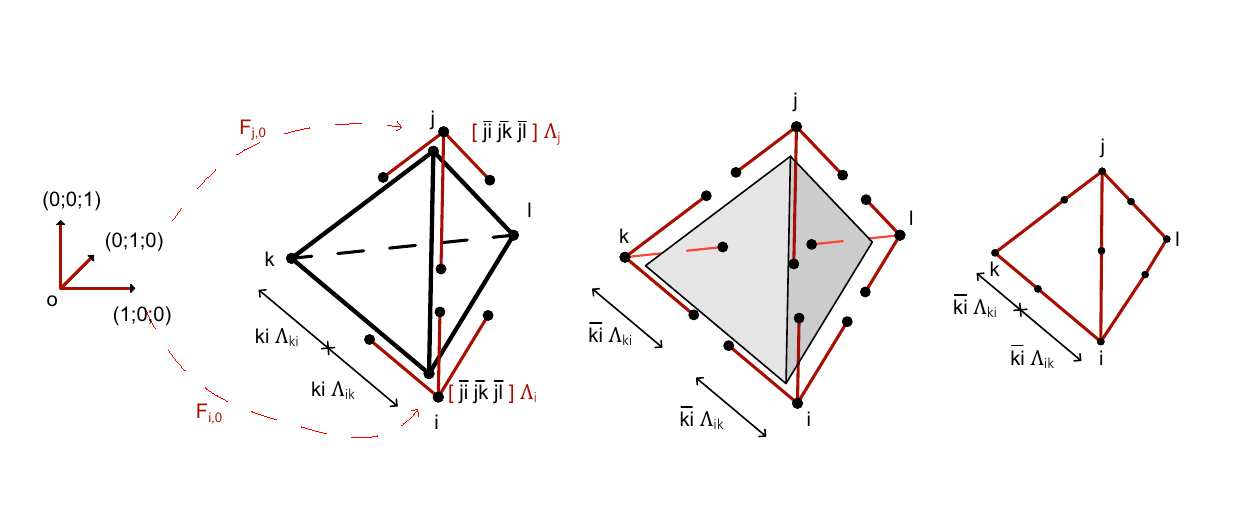} \caption{Cover of $H(P)$ by neighbourhoods.}\label{fig:localchart}
\end{figure}

\subsection{Adapted charts}

Let $P$ be a simple convex polyhedron and let $\{F_{j,\sigma,\bl}\}_{j\in I(P)}$ be a chart of $H(P)$ with permutations $\sigma=(\sigma_j)_{j\in I(P)}$ and rates $\bl=(\lambda_{ij})_{(i,j)\in E(P)}$. We can use a chart effectively during the procedure of chipping if the closest cutting points on the edges of a vertex belong to the neighbourhood of the vertex. In the following, we will define how the chart of $P$ adapted to the chipping $\mathcal{C}_{\underline{p}}(P)$ in such a way.

\begin{definition}[Adapted charts]
 Let $P$ be a simple convex polyhedron, and let $\underline{p}=(p_{ij})_{(i,j)\in E(P)}$ be a chipping rate. Let $\mathcal{F}_{\sigma',\bl'}'=\{F_{j',\sigma',\bl'}\}_{j'\in I(P)}$ be a chart of $H(P)$, and let $\mathcal{F}_{\sigma,\bl}=\{F_{j,\sigma,\bl}\}_{j\in I(\mathcal{C}_{\underline{p}}(P))}$ be a chart of $H(\mathcal{C}_{\underline{p}}(P))$.
 
 We say that the chart $\mathcal{F}_{\sigma,\bl}$ is {\normalfont adapted} to $\mathcal{F}_{\sigma',\bl'}'$ if for every $j\in I(\mathcal{C}_{\underline{p}}(P))$ with mother $j'\in I(P)$ with and father $i'\in I(P)$ such that $N(j')=\{i',k',l'\}$, we have that $N(j)=\{i'j',j'k',j'l'\}$ and $\sigma_{j}((\sigma_{j'}')^{-1})(i')=i'j'$, $\sigma_{j}((\sigma_{j'}')^{-1})(k')=j'k'$ and $\sigma_{j}((\sigma_{j'}')^{-1})(l')=j'l'$, moreover,
	\begin{multline*}
	\Lambda_{j',\sigma',\bl'}=\begin{bmatrix}
		p_{j'\sigma'_{j'}(1)} & 0 & 0 \\
		0 & p_{j'\sigma'_{j'}(2)} & 0 \\
		0 & 0 & p_{j'\sigma'_{j'}(3)}
	\end{bmatrix}+\\
\begin{bmatrix}
	1-p_{\sigma'_{j'}(1)j'}-p_{j'\sigma'_{j'}(1)} & 0 & 0 \\
	0 & 1-p_{\sigma'_{j'}(2)j'}-p_{j'\sigma'_{j'}(2)} & 0 \\
	0 & 0 & 1-p_{\sigma'_{j'}(3)j'}-p_{j'\sigma'_{j'}(3)}
\end{bmatrix}\Lambda_{j,\sigma,\bl}.
	\end{multline*}
\end{definition}

In particular, $\sigma$ being adapted to $\sigma'$ means that $\sigma$ gives the same position to the neighbours of $j$ as the permutation $\sigma'$ of the mother vertex of $j$ to the father vertices. Clearly, if the chart is adapted to the chipping, then for every $i,j\in I(P)$ with $(i,j)\in E(P)$, $p_{ij}<\lambda_{ij}$.

For a simple convex polyhedron $P$ and for $j,i\in I(P)$ with $(i,j)\in E(P)$, let us define a {\it sibling sequence} with respect to a chipping sequence $P_{n}=\mathcal{C}_{\underline{p}_n}(P_{n-1})$ with $P_0=P$ and chipping rates $\underline{p}=(p_{ij})_{(i,j)\in E(P_n)}$ as follows: let $j_0=j$ and $i_0=i$. By induction, if $j_n,i_n\in I(P_n)$ such that $(i_n,j_n)\in E(P_n)$ is defined then let $j_{n+1}=j_ni_{n}\in I(P_{n+1})$ and $i_{n+1}=i_{n}j_n\in I(P_{n+1})$, and by the definition of the chipping, $(i_{n+1},j_{n+1})\in E(P_{n+1})$. Clearly, the sibling of $j_n$ is $s(j_n)=i_n$ for every $n\geq1$. For the first two steps of the sibling sequence and the length of the edge between them, see Figure~\ref{fig:sibling}.

\begin{figure}[h]
	\centering
	\includegraphics[width=10cm]{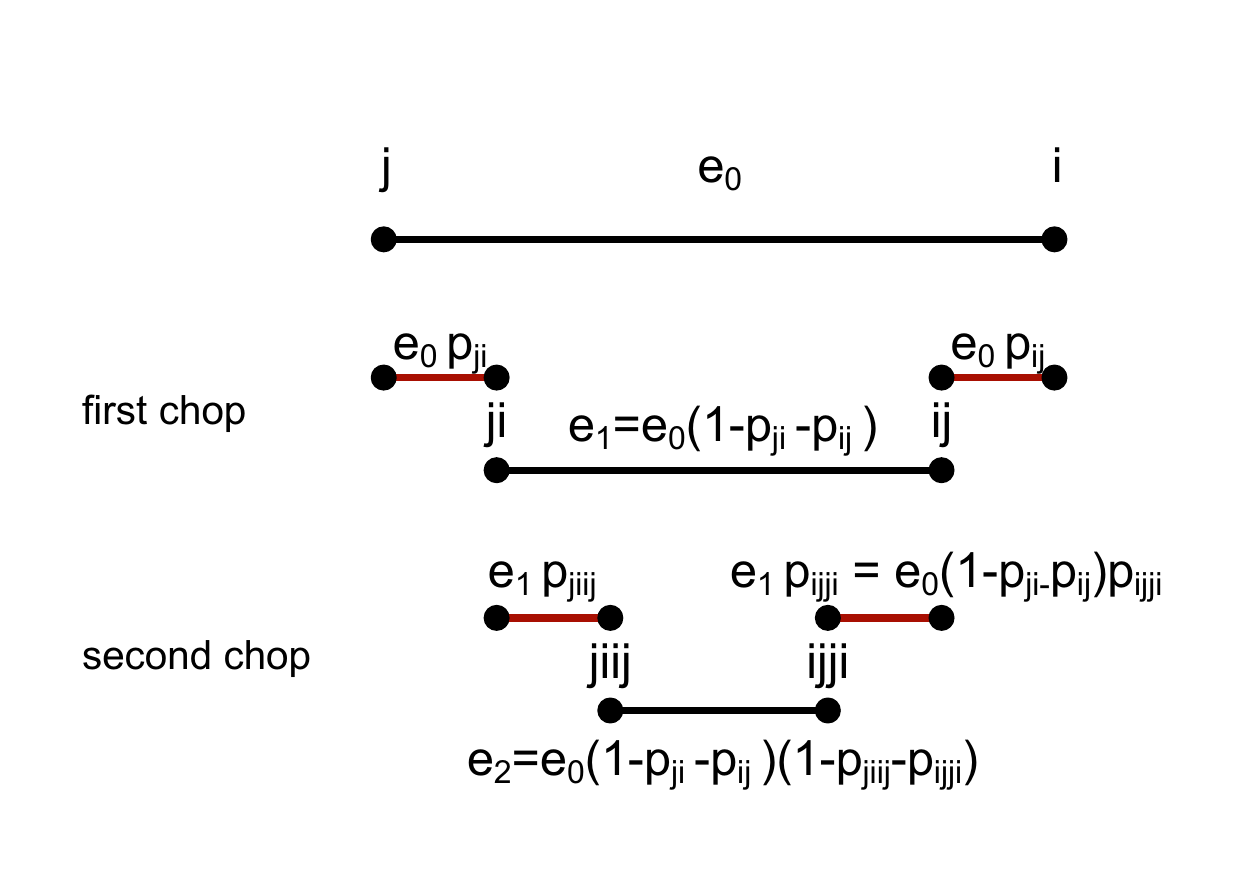} \caption{Effect of chipping on one edge}\label{fig:sibling}
\end{figure}

\begin{lemma}\label{lem:adaptedchart is unique}
	Let $P$ be a simple convex polyhedron, and let $P_{n}=\mathcal{C}_{\underline{p}_n}(P_{n-1})$ be a chipping sequence with $P_0=P$ and chipping rates $\underline{p}_n=(p_{ij})_{(i,j)\in E(P_{n-1})}$. Let $\mathcal{F}^{(n)}_{\sigma^{(n)},\bl^{(n)}}$ be charts of $H(P_n)$ such that $\mathcal{F}^{(n)}_{\sigma^{(n)},\bl^{(n)}}$ is adapted to $\mathcal{F}^{(n-1)}_{\sigma^{(n-1)},\bl^{(n-1)}}$ for every $n\geq1$. Then for every $j,i\in I(P)$ with $(i,j)\in E(P)$ we have
	$$
	\lambda_{ji}=\sum_{k=1}^\infty p_{j_k}\prod_{\ell=1}^{k-1}(1-p_{j_\ell}-p_{s(j_\ell)}),
	$$
	where $j_n$ is the sibling sequence with initial vertices $j$ and $i$.
\end{lemma}

\begin{proof}
	By definition of adaptedness of charts,
	$$
	\lambda_{ji}=p_{ji}+(1-p_{ji}-p_{ij})\lambda_{jiij}=p_{j_1}+(1-p_{j_1}-p_{s(j_1)})\lambda_{j_1i_1}.
	$$
	The proof can be finished now by induction.
\end{proof}

For a chipping sequence $P_{n}=\mathcal{C}_{\underline{p}_n}(P_{n-1})$, if $\mathcal{F}^{(n)}_{\sigma^{(n)},\bl^{(n)}}$ are charts of $H(P_n)$ such that $\mathcal{F}^{(n)}_{\sigma^{(n)},\bl^{(n)}}$ is adapted to $\mathcal{F}^{(n-1)}_{\sigma^{(n-1)},\bl^{(n-1)}}$ for every $n\geq1$ then we say that the {\it sequence of charts is adapted}. By the definition of adaptedness, $\sigma^{(0)}$ uniquely defines every permutation sequence $\sigma^{(n)}$ for every $n\geq1$.

A simple corollary of Lemma~\ref{lem:adaptedchart is unique} is that the chart $\mathcal{F}_{\sigma,\bl}$ of $H(P_0)$ is uniquely determined by the chipping sequence $P_{n}=\mathcal{C}_{\underline{p}_n}(P_{n-1})$ and the vector of permutations $\sigma=(\sigma_j)_{j\in I(P_0)}$, and hence, it defines uniquely the sequence of adapted charts. 

\subsection{Construction of the iterated function {scheme representation of chipping}}\label{subsec:ifs}

For a simple convex polyhedron $P$, chipping rates $\underline{p}=(p_{i'j'})_{(i',j')\in E(P)}$ and permutations $\sigma=(\sigma_{j})_{j\in I(\mathcal{C}_{\underline{p}}(P))}$, where $\sigma_{j}$ is a permutation of $N(j)$, let us define for every $j\in I(\mathcal{C}_{\underline{p}}(P))$ a $3\times3$ matrix $C_{j,\sigma,\underline{p}}$ such that
\begin{equation}\label{eq:Cmatrix}
	C_{j,\sigma,\underline{p}}=\begin{cases}\begin{bmatrix}
			1-p_{j}-p_{\sigma(1)} & -p_j & -p_j \\
			0 & p_{\sigma(2)} & 0 \\
			0 & 0 & p_{\sigma(3)} \\
		\end{bmatrix} & \text{if $\sigma(1)$ is the sibling of $j$,}\vspace{5pt}\\
		\begin{bmatrix}
			p_{\sigma(1)} & 0 & 0 \\
			-p_{j} & 1-p_{\sigma(2)}-p_j & -p_j \\
			0 & 0 & p_{\sigma(3)} \\
		\end{bmatrix} & \text{if $\sigma(2)$ is the sibling of $j$, and}\vspace{5pt}\\
		\begin{bmatrix}
			p_{\sigma(1)} & 0 & 0 \\
			0 & p_{\sigma(2)} & 0 \\
			-p_{j} & -p_j & 1-p_{\sigma(3)}-p_j \\
		\end{bmatrix} & \text{if $\sigma(3)$ is the sibling of $j$.}\end{cases}
\end{equation}

%Furthermore, if $\sigma_1$ is adapted to $\sigma_0$, we get that for any $j\in\mathcal{C}_{\underline{p}}(P)$ with mother $j'$ and father $i'$, we get $C_{j,\sigma_1}=C_{j'i',\sigma_0}$.

\begin{lemma}\label{lem:chart}
	Let $P$ be a simple convex polyhedron, and let $\underline{p}=(p_{ij})_{(i,j)\in E(P)}$ be a vector of chipping rates. Furthermore, let $\mathcal{F}_{\sigma,\bl}$ and $\mathcal{F}'_{\sigma',\bl'}$ be charts of $H(\mathcal{C}_{\underline{p}}(P))$ and $H(P)$ respectively such that $\mathcal{F}_{\sigma,\bl}$ is adapted to $\mathcal{F}'_{\sigma',\bl'}$.
	
	Then for every $j\in I(\mathcal{C}_{\underline{p}}(P))$ with mother $j'\in I(P)$ we have
	$$
	F_{j,\sigma,\bl}=F_{j',\sigma',\bl'}\circ G_{j',j,\sigma',\sigma,\underline{p}},
	$$
	where $G_{j',j,\sigma',\sigma,\underline{p}}(x)=\Lambda_{j',\sigma',\bl'}^{-1}\left(C_{j,\sigma,\underline{p}}\Lambda_{j,\sigma,\bl}x+p_je_{\sigma^{-1}(s(j))}\right)$.
\end{lemma}

\begin{proof}
	Let $j\in I(\mathcal{C}_{\underline{p}}(P))$ be arbitrary. Let $j'\in I(P)$ the mother, and $i'\in N(j')$ be the father of $j$, i.e. $j=j'i'$. Let us denote the other neighbours of $j'$ by $N(j')\setminus\{i'\}=\{k',l'\}$. Then the neighbours of $j=j'i'$ are $i:=i'j'$, $k:=j'k'$ and $l:=j'l'$. In particular, $i'j'=s(j'i')=s(j)$. Hence,
	\begin{eqnarray*}
	A_j&=&A_{j'i'}=A_{j'}+p_{j'i'}\overrightarrow{A_{j'}A_{i'}},\\
	A_i&=&A_{i'j'}=A_{i'}+p_{i'j'}\overrightarrow{A_{i'}A_{j'}},\\
	A_k&=&A_{j'k'}=A_{j'}+p_{j'k'}\overrightarrow{A_{j'}A_{k'}}\text{ and}\\
	A_l&=&A_{j'l'}=A_{j'}+p_{j'l'}\overrightarrow{A_{j'}A_{l'}}.
	\end{eqnarray*}
This implies that
\begin{equation}\label{eq:h1}
	\begin{split}
\overrightarrow{A_jA_i}&=(1-p_{i'j'}-p_{j'i'})\overrightarrow{A_{j'}A_{i'}},\\
\overrightarrow{A_jA_k}&=p_{j'k'}\overrightarrow{A_{j'}A_{k'}}-p_{j'i'}\overrightarrow{A_{j'}A_{i'}},\\
\overrightarrow{A_jA_l}&=p_{j'l'}\overrightarrow{A_{j'}A_{l'}}-p_{j'i'}\overrightarrow{A_{j'}A_{i'}}.
\end{split}
\end{equation}
Thus, since the charts are adapted, we have $\sigma^{-1}_{j}(i)=\sigma^{-1}_{j'}(i')$, $\sigma^{-1}_{j}(k)=\sigma^{-1}_{j'}(k')$ and $\sigma^{-1}_{j}(l)=\sigma^{-1}_{j'}(l')$. So by \eqref{eq:h1} and the definition of matrix $C_{j,\sigma,\underline{p}}$ \eqref{eq:Cmatrix}, we get
\begin{equation}\label{eq:adapted1}
	\boldsymbol{A}_{j,\sigma}=\boldsymbol{A}_{j',\sigma'}C_{j,\sigma,\underline{p}}.
\end{equation}
Since $i=i'j'$ is the sibling of $j=j'i'$ we get by \eqref{eq:adapted1} that
\[
\begin{split}
		F_{j,\sigma,\bl}(x)&=\boldsymbol{A}_{j,\sigma}\Lambda_{j,\sigma,\bl}x+A_j\\
&=\boldsymbol{A}_{j',\sigma}C_{j,\sigma,\underline{p}}\Lambda_{j,\sigma,\bl}x+A_{j'}+p_{j'i'}\overrightarrow{A_{j'}A_{i'}}\\
&=\boldsymbol{A}_{j',\sigma}\left(C_{j,\sigma,\underline{p}}\Lambda_{j,\sigma,\bl}x+p_{j'i'}\underline{e}_{\sigma_{j'}^{-1}(i')}\right)+A_{j'}\\
&=F_{j',\sigma',\bl'}\left(\Lambda_{j',\sigma',\bl'}^{-1}C_{j,\sigma,\underline{p}}\Lambda_{j,\sigma,\bl}x+p_{j'i'}\Lambda_{j',\sigma',\bl'}^{-1}\underline{e}_{\sigma^{-1}_{j}(i)}\right),
\end{split}
\]
which had to be proven.
\end{proof}

Under the conditions of Lemma~\ref{lem:chart}, the adaptedness of the charts (i.e. $\sigma^{-1}_{j}(i)=\sigma^{-1}_{j'}(i')$, $\sigma^{-1}_{j}(k)=\sigma^{-1}_{j'}(k')$ and $\sigma^{-1}_{j}(l)=\sigma^{-1}_{j'}(l')$) implies that
\begin{equation}\label{eq:images}
\begin{split}
G_{j',j,\sigma',\sigma,\underline{p}}(\underline{e}_{\sigma^{-1}_j(s(j))})&=\underline{e}_{\sigma^{-1}_j(s(j))},\\
G_{j',l,\sigma',\sigma,\underline{p}}(\underline{e}_{\sigma_j^{-1}(k)})&=G_{j',k,\sigma',\sigma,\underline{p}}(\underline{e}_{\sigma_j^{-1}(l)})\text{ for }k\neq l\in N(j)\setminus\{s(j)\}.
\end{split}
\end{equation}
Hence, 
\begin{equation}\label{eq:cont}
	G_{j',j,\sigma',\sigma,\underline{p}}(W)\subset W,
	\end{equation} where $W$ is the tetrahedron defined by the vectors $\{0,e_1,e_2,e_3\}$.

\subsection{Proof of the existence of the limiting object}\label{sec:existence}

Let now $P_{n}=\mathcal{C}_{\underline{p}_n}(P_{n-1})$ be a chipping sequence with $P_0=P$ and with chipping rates $\underline{p}_n=(p_{ij})_{(i,j)\in E(P_n)}$. Let $\mathcal{F}^{(n)}_{\sigma^{(n)},\bl^{(n)}}$ be the uniquely determined sequence of adapted charts of $H(P_n)$. For every $n\geq1$ and $j_n\in I(P_n)$, there exists a unique sequence $j_k\in I(P_k)$ such that $j_k$ is the mother of $j_{k+1}$ for every $k=0,\ldots,n-1$. We call the sequence $\underline{j}=(j_k)_{k=0}^{n-1}$ the {\it mother sequence} of $j_n$. Let us denote the set of infinite mother sequences by $\Sigma$, that is,
$$
\Sigma=\{(j_0,j_1,\ldots):j_k\in I(P_k)\text{ and }j_{k-1}\text{ is the mother of }j_k\}.
$$
Furthermore, denote $\Sigma_n$, the set of mother sequences of length $n$ and denote $\Sigma_*$, the set of finite mother sequences. For a $\underline{j}\in\Sigma_*$, denote $|\underline{j}|$ the length of $\underline{j}$, that is, $|\underline{j}|=n$ for $\underline{j}=(j_0,\ldots,j_n)$.

By applying Lemma~\ref{lem:chart} inductively, we get
\begin{equation}\label{eq:Fiterate}
\begin{split}
F_{j_n,\sigma^{(n)},\bl^{(n)}}&=F_{j_{n-1},\sigma^{(n-1)},\bl^{(n-1)}}\circ G_{j_{n-1},j_{n},\sigma^{(n-1)},\sigma^{(n)},\underline{p}_n}\\
&=F_{j_n,\sigma^{(0)},\bl^{(0)}}\circ G_{j_0,j_1,\sigma^{(0)},\sigma^{(1)},\underline{p}_1}\circ\cdots\circ G_{j_{n-1},j_{n},\sigma^{(n-1)},\sigma^{(n)},\underline{p}_n}.
\end{split}
\end{equation}
For a mother sequence $\underline{j}=(j_0,\ldots,j_n)$ and a sequence chipping rates $\boldsymbol{p}=(\underline{p}_1,\ldots,\underline{p}_n)$ let
\begin{equation}\label{eq:mapG}
G_{\underline{j},\sigma^{(0)},\bp}:=G_{j_0,j_1,\sigma^{(0)},\sigma^{(1)},\underline{p}_1}\circ\cdots\circ G_{j_{n-1},j_{n},\sigma^{(n-1)},\sigma^{(n)},\underline{p}_n}.
\end{equation}
Furthermore, let 
\begin{equation}\label{eq:Cmatrix2}
C_{\underline{j},\sigma^{(0)},\bp}:=C_{j_1,\sigma^{(1)},\underline{p}_1}\cdots C_{j_n,\sigma^{(n)},\underline{p}_n}.
\end{equation}
Note that $\sigma^{(0)}$ determines uniquely the further permutations, so the product and composition above depend only on it. Moreover, for an integer $1\leq k\leq n-1$ let $\underline{j}|_k=(j_1,\ldots,j_k)$ and $\bp|_k=(\underline{p}_1,\ldots,\up_k)$. For $k=0$, we use the conventions $\underline{j}|_0=\emptyset$, $\bp|_0=\emptyset$, $C_{\emptyset,\emptyset}=\mathrm{Id}$.

It is easy to see by the definition of the map $G_{\underline{j},\sigma^{(0)},\bp}$ that
\begin{multline*}
G_{\underline{j},\sigma^{(0)},\bp}(x)=\\
\Lambda_{j_0,\sigma^{(0)},\bl^{(0)}}^{-1}\left(C_{\underline{j},\sigma^{(0)},\bp}\ \Lambda_{j_n,\sigma^{(n)},\bl^{(n)}}x+\sum_{k=1}^np_{j_k}C_{\underline{j}|_{k-1},\sigma^{(0)},\bp|_{k-1}}\underline{e}_{\sigma_{j_k}^{-1}(s(j_k))}\right).
\end{multline*}

Let us denote the set of indexes of vertices through the chipping by $I^*:=\bigcup_{n=0}^\infty I(P_n)$. 

Let us denote the singular values of a real $3\times3$ matrix $A$ by $\alpha_1\geq\alpha_2\geq\alpha_3$. Clearly, for a $3\times3$ matrix $A$, $\|A\|=\alpha_1(A)$, where $\|.\|$ is the induced norm by the usual Euclidean norm, and $\alpha_3(A)=\|A^{-1}\|^{-1}$. Furthermore, let $\|.\|_1$ be the $1$-norm of $\R^3$, that is, for $\underline{v}^T=(v_1,v_2,v_3)\in\R^3$ $\|\underline{v}\|_1=|v_1|+|v_2|+|v_3|$. With a slight abuse of notation, let us denote the $1$-norm of $3\times3$ matrices by $\|.\|_1$ too, that is, for a $3\times3$ matrix $A=(a_{ij})_{i,j=1}^3$, we have $\|A\|_1=\max_{1\leq j\leq3}\{|a_{1j}|+|a_{2j}|+|a_{3j}|\}$.

\begin{lemma}\label{lem:contr}
	If the chipping rates $\bp$ are regular, then there exists $C>0$ such that for every $n\geq1$ and every mother sequence $\underline{j}=(j_0,\ldots,j_{n-1})$ of length $n$
	$$
	\alpha_1(C_{\underline{j},\sigma^{(0)},\bp})\leq C(1-\delta)^n.
	$$
\end{lemma}

\begin{proof}

	Since every two norms over finite dimensional vector spaces are equivalent, there exists $C>0$ such that
	$$
	\alpha_1(C_{\underline{j},\sigma^{(0)},\bp})\leq C\|C_{\underline{j},\sigma^{(0)},\bp}\|_1\leq C\prod_{k=1}^n\|C_{j_k,\sigma^{(0)},\underline{p}_k}\|_1.
	$$
	
	On the other hand, for every matrix $C_{j,\sigma,\underline{p}}$ defined in \eqref{eq:Cmatrix}
	$$
	\|C_{j,\sigma,\underline{p}}\|_1=\max\left\{1-p_j-p_{s(j)},\max_{k\in N(j)\setminus\{s(j)\}}\{p_j+p_k\}\right\}\leq 1-\delta,
	$$
	which follows from the regularity of the rates.
\end{proof}

Now, we are ready to show the existence of the limit set of the sequence of net edges.

\begin{proof}[Proof of Proposition~\ref{prop:limitset}]
	Let $P$ be a simple convex polyhedron. Let $\bp=(\underline{p}_n)_{n=1}^\infty$ be a sequence of regular chipping rates and let $P_n=\mathcal{C}_{\underline{p}_n}(P_{n-1})$ with $P_0=P$ be a chipping sequence.
	
	Let us fix a permutation vector $\sigma^{(0)}=(\sigma_j)_{j\in I(P)}$ and let $\mathcal{F}^{(n)}$ be a sequence of adapted charts of $H(P_n)$. Note that $\bp$ and $\sigma$ uniquely determine the sequence of adapted charts, so by a slight abuse of notation, we omit the fixed permutation from the notations.
	
	Let $W$ be the tetrahedron defined by the vectors $\{\underline{0},\underline{e}_1,\underline{e}_2,\underline{e}_3\}$, and let $L=\bigcup_{i=1}^3[\underline{0},\underline{e}_i]$, where $[\underline{0},\underline{e}_i]=\{t\underline{e}_i:t\in[0,1]\}$. By definition, $H(P_n)=\bigcup_{F\in\mathcal{F}^{(n)}}F(L)\subset \bigcup_{F\in\mathcal{F}^{(n)}}F(W)=:X_n$.
	
	First, we show that $X_n$ converges to a limit set $X$ as $n\to\infty$. By Lemma~\ref{lem:chart} and \eqref{eq:Fiterate}, for every $F\in\mathcal{F}^{(n)}$, there exists a mother sequence $\underline{j}$ such that $F=F'\circ G_{\underline{j}|_n,\bp|_n}$, where $F'\in\mathcal{F}^{(0)}$. Thus, $X_n\subseteq X_{n-1}$ by \eqref{eq:cont}. Clearly, $X_n$ are compact sets, and so there exists a non-empty compact set
	\begin{equation}\label{eq:ifsconst}
		X:=\bigcap_{n=1}^\infty X_n=\bigcap_{n=1}^\infty\bigcup_{F\in\mathcal{F}^{(n)}}F(W).
	\end{equation}
	 Furthermore, by Lemma~\ref{lem:contr},
	\[
	\begin{split}
	&d_H(F'\circ G_{\underline{j}|_{n-1},\bp|_{n-1}}(W), F'\circ G_{\underline{j}|_n,\bp|_n}(W))\\
	&\qquad\leq\alpha_1(\boldsymbol{A}_{j_0}C_{\underline{j}|_{n-1},\bp|_{n-1}}\Lambda_{j_{n-1},\bl^{(n-1)}})d_H(W,G_{j_{n-1},j_n,\underline{p}_n}(W))\\
	&\qquad\leq\alpha_1(\boldsymbol{A}_{j_0})C(1-\delta)^{n-1}\mathrm{diam}(W),
\end{split}
\]
where recall that $\boldsymbol{A}_{j_0}=\begin{bmatrix}\overrightarrow{A_{j_0}A_{\sigma_{j_0}(1)}} & \overrightarrow{A_{j_0}A_{\sigma_{j_0}(2)}} & \overrightarrow{A_{j_0}A_{\sigma_{j_0}(3)}}\end{bmatrix}$, and so
$$
d_H(X_{n-1},X_n)\leq(1-\delta)^{n-1}C\mathrm{diam}(W)\max_{j\in I(P)}\alpha_1(\boldsymbol{A}_{j}).
$$
This implies that $X_n$ forms a Cauchy sequence, and so $X_n\to X$ in the Hausdorff metric. Finally,
\[\begin{split}
d_H(H(P_n),X_n)&\leq\max_{F\in\mathcal{F}^{(n)}}d_H(F(L),F(W))\leq\max_{F\in\mathcal{F}^{(n)}}\mathrm{diam}(F(W))\\
&\leq\max_{j\in I(P)}\alpha_1(\boldsymbol{A}_{j})C(1-\delta)^{n-1}\mathrm{diam}(T),
\end{split}\]
	which completes the proof.
\end{proof}

\subsection{The box-counting dimension}\label{sec:boxdim}

The main purpose of this paper is to study the fractal properties of the limit set $X$ defined in Proposition~\ref{prop:limitset}. {There is no widely
	accepted definition of fractals. However, most authors would call a set	fractal if at infinitely many scales its smaller parts resemble the whole.} As we have seen by the construction \eqref{eq:ifsconst} in Section~\ref{sec:existence}, $X$ is a finite union of such sets since we repeat the same kind of chipping again and again. In particular, if the chipping rates $p_{ji}\equiv p$ were taken as a constant value independent of $n$ and $i,j\in I(P_n)$, then the limiting object $X$ would be a finite union of self-affine sets. More precisely, $X=\bigcup_{F'\in\mathcal{F}^{(0)}}F'(Y)$, where $Y$ is the unique non-empty compact set such that 
\begin{equation}\label{eq:constantchipping}
Y=\left(C_1Y+2p\underline{e}_1\right)\cup \left(C_2Y+2p\underline{e}_2\right)\cup \left(C_3Y+2p\underline{e}_3\right),
\end{equation} where 
$C_i$ is the matrix defined in \eqref{eq:Cmatrix} such that $j$ has sibling $\sigma(i)$.

Let us now define the box-counting dimension. Let $A$ be a bounded subset of $\R^3$. Let $N_{\delta}(A)$ be the minimal number of balls that cover $A$. We define the {\it upper box-counting dimension} as
$$
\overline{\dim}_B(A)=\limsup_{\delta\to0}\frac{\log N_\delta(A)}{-\log\delta}.
$$

Two useful properties of the upper box-counting dimension are finite stability and monotonicity under Lipschitz mappings. That is, for every finite index set $J$
$$
\overline{\dim}_B\left(\bigcup_{j\in J}A_j\right)=\max\{\overline{\dim}_B(A_j):j\in J\},
$$
and for every map $f\colon\R^3\mapsto\R^3$ such that $\|f(\underline{x})-f(\underline{y})\|\leq L\|\underline{x}-\underline{y}\|$ for every $\underline{x},\underline{y}\in\R^3$ with some uniform constant $0<L$
$$
\overline{\dim}_B\left(f(A)\right)\leq\overline{\dim}_B(A).
$$
For further properties of the box-counting dimension, see \cite[Section~3.1]{Falb2}.

The difficulty which arises in the calculation of the dimension of self-affine-like objects is that they are defined by strict affine mappings. In other words, the most natural cover of $X$ is the collection $\{F(W)\}_{F\in\mathcal{F}^{(n)}}$ by the construction \eqref{eq:ifsconst}, but the sets $F(W)$ are relatively long and thin shapes which do not fit the required cover by balls. To handle this difficulty, let us define the {\it singular value function} introduced by Falconer \cite{Fal1}. For a $3\times3$ matrix $A$, let 
$$
\varphi^s(A)=\begin{cases}
	\alpha_1(A)^s & \text{if }0\leq s\leq 1\\
	\alpha_1(A)\alpha_2(A)^{s-1} & \text{if }1< s\leq2\\
	\alpha_1(A)\alpha_2(A)\alpha_3(A)^{s-2} & \text{if }2< s\leq3\\
	\big(\alpha_1(A)\alpha_2(A)\alpha_3(A)\big)^{s/3} & \text{if }s>3,
\end{cases}
$$
where $\alpha_i(A)$ denotes the $i$th singular value of $A$. The function $s\mapsto\varphi^s(A)$ is monotone decreasing, and the function $A\mapsto\varphi^s(A)$ is sub-multiplicative, i.e. for every $3\times3$ matrices $A,B$, $\varphi^s(AB)\leq\varphi^s(A)\varphi^s(B)$ for every $s\geq0$, see \cite{Fal1}.

\begin{theorem}\label{thm:main}
	Let $P$ be a convex polyhedron, and let $P_{n}=\mathcal{C}_{\underline{p}_n}(P_{n-1})$ be a chipping sequence with $P_0=P$ and chipping rates $\underline{p}_n=(p_{i})_{i\in I(P_{n})}$. {Let $X$ be the limiting object of the sequence $H(P_n)$ defined in Proposition~\ref{prop:limitset}.} Let $\sigma^{(0)}$ be an arbitrary but fixed neighbourhood permutation of $P$. Furthermore, let $\{C_{\underline{j},\bp}\}$ be the matrices defined in \eqref{eq:Cmatrix} and \eqref{eq:Cmatrix2}.
	
	If there exists $\delta>0$ such that for every $j\in I^*=\bigcup_{n=0}^\infty I(P_n)$ and $i\in N(j)$, $p_j>\delta$ and $p_{j}+p_{i}<1-\delta$ then	
	$$
	\overline{\dim}_B(X)=s_0:=\inf\{s>0:\sum_{n=1}^\infty\sum_{\underline{j}\in\Sigma_n}\varphi^s(C_{\underline{j},\bp|_n})<\infty\}.
	$$
	Furthermore, $\overline{\dim}_B(X)\in[1,2]$.
\end{theorem}
	
The method of the proof uses the ideas of Falconer~\cite{Fal2}. However, there are several technical difficulties; for instance, $X$ is not a planar, connected set. 

{The dimension estimate $1\leq\overline{\dim}_B(X)\leq2$ seems virtually obvious, since the set $X$ is connected and its projection to the plane with normal vector $(1,1,1)$ satisfies a certain separation condition (see Lemma~\ref{lem:projosc} later). However, it is natural to expect that $1<\overline{\dim}_B(X)<2$, which would be an even stronger fractal-like property and which might be possible to show with some more sophisticated analysis relying submultiplicativity and the mentioned separation.}

\subsection{The Hausdorff dimension}

Let us also mention another important dimension concept, the Hausdorff dimension. For a subset $A$ of $\R^d$, we define the Hausdorff dimension of $A$ by
\begin{multline*}
	\dim_H(A)=\inf\Big\{s>0:\text{there exists $\{U_i\}_{i\in\mathcal{I}}$ such that }\\
	A\subseteq\bigcup_{i\in\mathcal{I}}U_i\text{ and }\sum_{i\in\mathcal{I}}\mathrm{diam}(U_i)^s<\infty\Big\}.
\end{multline*}
It is clear from the definition that $\dim_H(A)\leq\overline{\dim}_B(A)$. Moreover, the Hausdorff dimension is also monotone under Lipschitz maps, but it is countably stable. For further properties, see Falconer \cite{Falb2}. 

In the case of the highly general chipping model with regular rates, the calculation of the Hausdorff dimension of the limiting object seems unattainable even with the most current techniques. However, the constant chipping rate case \eqref{eq:constantchipping} can be handled thanks to the recent work of Rapaport \cite{R} and Morris and Sert \cite{MS}. Our second main theorem gives the Hausdorff dimension of the net of edges of the abraded polyhedron with constant chipping rate.

\begin{theorem}\label{thm:main2}
	Let $P$ be a convex polyhedron and $0<p<1/2$. Let $P_{n}=\mathcal{C}_{\underline{p}_n}(P_{n-1})$ be a chipping sequence with $P_0=P$ and chipping rates $\underline{p}_n=(p_{i})_{i\in I(P_{n})}$ such that $p_i=p$ for every $i\in \bigcup_{n=0}^\infty I(P_n)$, and let $X$ be the limiting object of the sequence $H(P_n)$ defined in Proposition~\ref{prop:limitset}. Then	
	$$
	\dim_H(X)=s_0=\inf\{s>0:\sum_{n=1}^\infty\sum_{j_1,\ldots,j_n=1}^3\varphi^s(C_{j_1}\cdots C_{j_n})<\infty\},
	$$
	where $C_1,C_2$ and $C_3$ are the matrices from the equation \eqref{eq:constantchipping}.
\end{theorem}

Although the matrices are fixed throughout the construction, approximating the dimension value is extremely difficult, see Morris \cite{Morris22}, and out of the scope of this paper.

\section{The upper box dimension of the edge net of the abraded polyhedron}\label{sec:proof}

We assume, without loss of generality, that $P_0$ is a simple convex polyhedron. Throughout the section, we fix a $P_{n}=\mathcal{C}_{\underline{p}_n}(P_{n-1})$ chipping sequence with $P_0=P$ and chipping rates $\bp=(\underline{p}_n)_{n=1}^\infty$, where $\underline{p}_n=(p_{i})_{i\in I(P_{n})}=(p_{ij})_{(i,j)\in E(P_{n-1})}$, furthermore, we fix a neighbourhood permutation $\sigma=(\sigma_j)_{j\in I(P)}$ of $P_0$. We will also assume throughout the section that chipping rates are regular. That is, there exists a $\delta>0$ such that for every $i\in I^*=\bigcup_{n=0}^\infty I(P_n)$, $\delta<p_{i}$ and $p_{i}+p_{j}<1-\delta$ for every $j\in N(i)$. 

To simplify the notations, we denote the adapted sequence of charts by $\mathcal{F}^{(n)}=\{F_{j}\}_{j\in I(P_n)}$ and the matrix defined in \eqref{eq:Lambda} by $\Lambda_j$ for $j\in I^*$. Let us denote the matrices defined in \eqref{eq:Cmatrix} by $C_{j}$ for $j\in I^*$, and the matrices defined in \eqref{eq:Cmatrix2} by $C_{\underline{j}}$ for mother sequences $\underline{j}\in\Sigma_*$. Similarly, for a $j\in I^*$ with mother $j'$, let $G_{j',j}$ be the map defined in Lemma~\ref{lem:chart}, and for a mother sequence $\underline{j}\in\Sigma_*$, let $G_{\underline{j}}$ be defined in \eqref{eq:mapG}.

For a mother sequence $\underline{j}=(j_0,j_1,\ldots,j_n)\in\Sigma_n$, let $$B_{\underline{j}}:=\Lambda_{j_0}^{-1}C_{\underline{j}}\Lambda_{j_n}.$$ Simple algebraic calculations show that for every $j\in I^*$, the matrix $C_j^{-1}$ contains non-negative elements, hence, $(B_{\underline{j}})^{-1}$ has non-negative elements for every mother sequence $\underline{j}\in\Sigma_*$.

Let $s_0$ be as in Theorem~\ref{thm:main}. By Lemma~\ref{lem:adaptedchart is unique} and the regularity of the chipping rates $\bp$, we have that for every $j\in I^*$, $\delta<\lambda_j<1$, and so 
\begin{equation}\label{eq:s0isB}
s_0=\inf\{s>0:\sum_{\underline{j}\in\Sigma_*}\varphi^s(B_{\underline{j}})<\infty\}.
\end{equation}

\subsection{Singular values and separation}\label{sec:singvalsep}

Before we prove the lower bound in Theorem~\ref{thm:main}, we need further analysis of the matrices $B_{\underline{j}}$. For an $\varepsilon\geq0$, let us denote the triangle formed by the vertices $(1-2\varepsilon,\varepsilon,\varepsilon)$, $(\varepsilon,1-2\varepsilon,\varepsilon)$ and $(\varepsilon,\varepsilon,1-2\varepsilon)$ by $T_\varepsilon$. Let us denote the triangle formed by the vertices $\underline{e}_1,\underline{e}_2,\underline{e}_3$ by $T_0$. Let us denote the orthogonal projection to a proper subspace $Y$ of $\R^3$ by $\mathrm{proj}_Y$. Furthermore, let us denote the subspace perpendicular to $(1,1,1)$ by $V$ and, for simplicity, by $\mathrm{proj}$ the orthogonal projection $\mathrm{proj}\colon\R^3\mapsto V$. Also, let us denote the standard scalar product on $\R^3$ by $\langle\cdot,\cdot\rangle$, and the angle between two vectors by $\sphericalangle(\cdot,\cdot)$.
	
	\begin{lemma}\label{lem:estim1}
	For every $\varepsilon>0$, there exists a uniform constant $C>0$ such that for every $\underline{v}\in T_\varepsilon$, and every mother sequence $\underline{j}\in\Sigma_*$
	$$
	\|(B_{\underline{j}})^{-1}\underline{v}\|\geq C\|(B_{\underline{j}})^{-1}\|.
	$$
\end{lemma}

	\begin{proof}		
	Let $A$ be an arbitrary but fixed matrix with strictly positive elements such that for every $\underline{w}\in T_0$, $A\underline{w}/\|A\underline{w}\|_1\in T_\varepsilon$. Since $(B_{\underline{j}})^{-1}\underline{v}$ has non-negative elements, we have $\frac{A(B_{\underline{j}})^{-1}\underline{v}}{\|A(B_{\underline{j}})^{-1}\underline{v}\|_1}\in T_\varepsilon$. By \cite[Lemma~2.2]{BoMo}, there exists a constant $C'>0$ depending only on $\varepsilon>0$ such that
	$$
	\|A(B_{\underline{j}})^{-1}\underline{v}\|\geq C'\|A(B_{\underline{j}})^{-1}\|.
	$$
	But clearly, $\|A(B_{\underline{j}})^{-1}\underline{v}\|\leq\|A\|\|(B_{\underline{j}})^{-1}\underline{v}\|$ and $\|A(B_{\underline{j}})^{-1}\|\geq\|A^{-1}\|^{-1}\|(B_{\underline{j}})^{-1}\|$, thus by choosing $C=C'\|A^{-1}\|^{-1}\|A\|^{-1}$, the claim follows.
	\end{proof}
	
	An immediate corollary of Lemma~\ref{lem:estim1} is that for every $\underline{v}\in T_\varepsilon$, every $n\geq1$, every mother sequence $\underline{j}\in\Sigma_*$
	\begin{equation}\label{eq:estimate3}
	\alpha_3(B_{\underline{j}})\geq C\|B_{\underline{j}}|(B_{\underline{j}})^{-1}V^\perp\|,
	\end{equation}
	where $\|A|Y\|$ denotes the restricted norm of the $3\times3$ matrix to the subspace $Y\subset\R^3$, that is, $\|A|Y\|=\sup_{\underline{v}\in Y}\|A\underline{v}\|$. More generally, denote by $\alpha_i(A|V)$ the $i$th singular value of the linear mapping $A|_V$ from $V$ to $\mathrm{Im}(A|_V)$.

	\begin{lemma}\label{lem:vectors}
		There exists $c>0$ such that for every mother sequence $\underline{j}\in\Sigma_*$ and for every $\underline{w}\in V$
		$$
		\|\mathrm{proj}(B_{\underline{j}}\underline{w})\|\geq c\alpha_2(B_{\underline{j}})\|\underline{w}\|.
		$$
	\end{lemma}

\begin{proof}
	 Let $\underline{w}\in V$ be arbitrary. Then, for any vector $\underline{v}\in V^\perp$
	 $$
	 0=\langle\underline{w},\underline{v}\rangle=\left\langle B_{\underline{j}}\underline{w},\left((B_{\underline{j}})^{-1}\right)^T\underline{v}\right\rangle.
	 $$
	 Since $\left((B_{\underline{j}})^{-1}\right)^T$ has non-negative elements, we have that $\left((B_{\underline{j}})^{-1}\right)^T\underline{v}$ is contained in the first octant $\{(x,y,z):x,y,z\geq0\}$, and so there exists a positive $\delta>0$ such that $\sphericalangle(\underline{v},B_{\underline{j}}\underline{w})>\delta$ for every mother sequence $\underline{j}$ and every $\underline{w}\in V$. And so, there exists $c>0$ such that
	 $$
	 \|\mathrm{proj}(B_{\underline{j}}\underline{w})\|\geq c\|B_{\underline{j}}\underline{w}\|.
	 $$
	
	 Let $\underline{w}_1,\underline{w_2}\in V$ be such that $\langle \underline{w}_1,\underline{w}_2\rangle=0$, and let $Z$ be the parallelepiped formed by $\underline{w}_1,\underline{w_2},(B_{\underline{j}})^{-1}\underline{v}$. Hence,
	 \[
	 \begin{split}
	 \alpha_1(B_{\underline{j}})\|\underline{w}_1\|\|B_{\underline{j}}\underline{w}_2\|\|\underline{v}\|&\geq\|B_{\underline{j}}\underline{w}_1\|\|B_{\underline{j}}\underline{w}_2\|\|\underline{v}\|\geq\mathrm{Vol}(B_{\underline{j}}(Z))\\
	 &=\alpha_1(B_{\underline{j}})\alpha_2(B_{\underline{j}})\alpha_3(B_{\underline{j}})\mathrm{Vol}(Z)\\
	 &\geq \alpha_1(B_{\underline{j}})\alpha_2(B_{\underline{j}})\alpha_3(B_{\underline{j}})\|\underline{w}_1\|\|\underline{w}_2\|\|(B_{\underline{j}})^{-1}\underline{v}\|\sqrt{2}/2\\
	 &\geq C\alpha_1(B_{\underline{j}})\alpha_2(B_{\underline{j}})\|\underline{w}_1\|\|\underline{w}_2\|,
	 \end{split}
	 \]
	 where in the last inequality, we used Lemma~\ref{lem:estim1}. Hence, the claim follows.
\end{proof}
	
	\begin{lemma}\label{lem:shape}
		For every mother sequence $\underline{j}\in\Sigma_*$ and every vector $\underline{w}$ with strictly positive coordinates, there exists $c>0$ such that
		$$
		G_{\underline{j}}(\underline{0})+c\underline{w}\in G_{\underline{j}}(T_0).
		$$
		In particular, for every $2$-dimensional subspace $Y$ with a normal vector of strictly positive coordinates, $\mathrm{proj}_Y(G_{\underline{j}}(W))=\mathrm{proj}_Y(G_{\underline{j}}(T_0))$.
	\end{lemma}
	
	\begin{proof}
		There exists $c\in\mathbb{R}$ such that $G_{\underline{j}}(\underline{0})+c\underline{w}$ is contained in the hyperplane $G_{\underline{j}}(V+\underline{v})$, where $\underline{v}=(1/3,1/3,1/3)$. Indeed, if not, then $\underline{w}\in B_{\underline{j}}(V)$ and in particular, $(B_{\underline{j}})^{-1}\underline{w}\in V$. Since $(B_{\underline{j}})^{-1}\underline{w}$ has strictly positive coordinates this is impossible.
		
		Now, let us argue by contradiction and suppose that\linebreak $G_{\underline{j}}(\underline{0})+c\underline{w}\in G_{\underline{j}}(V+\underline{v})\setminus G_{\underline{j}}(T_0)$. Clearly,  $G_{\underline{j}}^{-1}(G_{\underline{j}}(\underline{0})+c\underline{w})=c(B_{\underline{j}})^{-1}\underline{w}\in (V+\underline{v})\setminus T_0$, but this is impossible since $(B_{\underline{j}})^{-1}\underline{w}$ has strictly positive elements.
		
		Finally, let $Y\subset\R^3$ be a $2$-dimensional {subspace} with normal vector $\underline{w}$. Since $\mathrm{proj}_Y(G_{\underline{j}}(\underline{0})+c\underline{w})=\mathrm{proj}_Y(G_{\underline{j}}(\underline{0}))$, the last claim follows.
	\end{proof}
	
	\begin{lemma}\label{lem:area}
		There exists a uniform constant $C>0$ such that for every mother sequence $\underline{j}\in\Sigma^*$
		$$
		\mathrm{Area}\left(\mathrm{proj}(G_{\underline{j}}(T_0))\right)\geq C\alpha_1(B_{\underline{j}})\alpha_2(B_{\underline{j}}).
		$$
	\end{lemma}

\begin{proof}
	Let $\underline{v}=(1/3,1/3,1/3)$ as usual. By Lemma~\ref{lem:shape}, there exists $c>0$ such that $G_{\underline{j}}(\underline{0})+c\underline{v}\in G_{\underline{j}}(T_0)$. Since $G_{\underline{j}}^{-1}(G_{\underline{j}}(\underline{0})+c\underline{v})=c(B_{\underline{j}})^{-1}\underline{v}\in T_0$, we have $\|c(B_{\underline{j}})^{-1}\underline{v}\|\leq1$. Thus, by \eqref{eq:estimate3} we get
	\begin{equation}\label{eq:interstep}
	c\|\underline{v}\|=\left\|B_{\underline{j}}\frac{(B_{\underline{j}})^{-1}\underline{v}}{\|(B_{\underline{j}})^{-1}\underline{v}\|}\right\|\|c(B_{\underline{j}})^{-1}\underline{v}\|\leq C\alpha_3(B_{\underline{j}}).
	\end{equation}
	Let us denote the height of $G_{\underline{j}}(W)$ with respect to the side $G_{\underline{j}}(T_0)$ by $m$. Then
	\[
	\begin{split}
	\alpha_1(B_{\underline{j}})\alpha_2(B_{\underline{j}})\alpha_3(B_{\underline{j}})\mathrm{Area(T_0)}\frac{\sqrt{3}}{18}&=\alpha_1(B_{\underline{j}})\alpha_2(B_{\underline{j}})\alpha_3(B_{\underline{j}})\mathrm{Vol}(W)\\
	&=\mathrm{Vol}(G_{\underline{j}}(W))=\mathrm{Area}(G_{\underline{j}}(T_0))m/6\\
	&\leq\mathrm{Area}(G_{\underline{j}}(T_0))c\|\underline{v}\|/6\\
	&\leq \mathrm{Area}(G_{\underline{j}}(T_0))\alpha_3(B_{\underline{j}})C/6,
	\end{split}
	\]
	where in the last inequality, we used \eqref{eq:interstep}.
\end{proof}
	
	\begin{lemma}\label{lem:projosc}
	For every $n\geq1$ and every mother sequences $\underline{j},\underline{j}'\in\Sigma_n$ such that $j_0=j'_0$ but $\underline{j}\neq\underline{j}'$, i.e. there exists $1\leq k\leq n$ such that $j_k\neq j_k'$ then
	$$\mathrm{proj}(G_{\underline{j}}(W^o))\cap \mathrm{proj}(G_{\underline{j}'}(W^o))=\emptyset,$$
	where we recall that $\mathrm{proj}$ denotes the orthogonal projection to the subspace $V$ of normal vector $\underline{v}=(1/3,1/3,1/3)$.
	\end{lemma}

	\begin{proof}
	Let $1\leq k\leq n$ be the smallest integer such that $j_k\neq j_k'$. Since $G_{i',i}(W)\subseteq W$ for every $i\in I^*$ with mother $i'$, it is enough to show that for every $2$-dimensional subspace $Y\subset\R^2$ with normal vector of strictly positive entries
		$$
		\mathrm{proj}_Y(G_{j_{k-1},j_k}(W^o))\cap\mathrm{proj}_Y(G_{j_{k-1}',j_k'}(W^o))=\emptyset,
		$$
		where $\mathrm{proj}_Y$ is the orthogonal projection to the subspace $Y$. Hence, by Lemma~\ref{lem:shape} it is enough to show that
		$$\mathrm{proj}_Y(G_{j_{k-1},j_k}(T_0))^o\cap\mathrm{proj}_Y(G_{j_{k-1}',j_k'}(T_0))^o=\emptyset.$$
		However, this {simply follows by the geometric position of the triangles $G_{j_{k-1},j_k}(T_0)$ and $G_{j_{k-1}',j_k'}(T_0)$. That is, each of the triangles have a vertex on different coordinate axis, the common vertex is positioned on the plane formed by these coordinate axis, and the third vertices of both of the triangles lie on the other two coordinate-planes.}
	\end{proof}

\begin{lemma}\label{lem:length}
	There exists a constant $C>0$ such that for every mother sequence $\underline{j}\in\Sigma_*$ there exists an $i\in\{2,3\}$ such that
	$$
	\|\proj(G_{\underline{j}}(\underline{e}_i))-\proj(G_{\underline{j}}(\underline{e}_1))\|\geq C\alpha_1(B_{\underline{j}}).
	$$
\end{lemma}

\begin{proof}
	Let us consider the singular value decomposition of the linear map $\proj B_{\underline{j}}|_V\colon V\mapsto V$. Namely, let $\underline{x}_1,\underline{x}_2$ and $\underline{y}_1,\underline{y}_2$ be orthonormal bases of $V$ such that $\proj B_{\underline{j}}\underline{x}_i=\alpha_i(\proj B_{\underline{j}}|V)\underline{y}_i$ for $i=1,2$. 
	
	Now, let us consider the exterior product $\bigwedge V$. Clearly, $\dim\bigwedge V=1$ and $\proj B_{\underline{j}}|_V$ induces a linear map $(\proj B_{\underline{j}}|_V)^\wedge$ on $\bigwedge V$ naturally by
	$$
	(\proj B_{\underline{j}}|_V)^\wedge(\underline{x}\wedge\underline{y})=(\proj B_{\underline{j}}\underline{x})\wedge(\proj B_{\underline{j}}\underline{y})=:d_{\underline{j}}\ \underline{x}\wedge\underline{y},
	$$
	where $d_{\underline{j}}=\alpha_1(\proj B_{\underline{j}}|V)\alpha_2(\proj B_{\underline{j}}|V)\in\R$. Since
	$\|\underline{x}\wedge\underline{y}\|$ is the area of the parallelogram formed by the vectors $\underline{x},\underline{y}\in V$, we get by Lemma~\ref{lem:area} that there exists a constant $C>0$ such that for every mother sequence $\underline{j}\in\Sigma_*$
	\begin{equation}\label{eq:bounds}
	\alpha_1(\proj B_{\underline{j}}|V)\alpha_2(\proj B_{\underline{j}}|V)\geq C\alpha_1(B_{\underline{j}})\alpha_2(B_{\underline{j}}).
	\end{equation}
	Furthermore, since $\alpha_2(B_{\underline{j}})\geq\inf_{\underline{w}\in V}\|\proj B_{\underline{j}}\underline{w}\|$ we have that
	\begin{equation}\label{eq:restrictedsingvalue}
	\alpha_1(\proj B_{\underline{j}}|V)\geq C\alpha_1(B_{\underline{j}}).
	\end{equation}
	
	For every vector $\underline{z}\in V$ such that $\sphericalangle(\underline{x}_1,\underline{z})<\pi/2-\varepsilon$, $\|\proj B_{\underline{j}}\underline{z}\|\geq\alpha_1(\proj B_{\underline{j}}|V)\|z\|\cos(\pi/2-\varepsilon)$. Since the angle between $\underline{e}_3-\underline{e}_1$ and $\underline{e}_2-\underline{e}_1$ is $\pi/3$, by choosing $\varepsilon=\pi/12$ we get that there is an $i\in\{2,3\}$ such that
	$$
	\|\proj B_{\underline{j}}(\underline{e}_i-\underline{e}_1)\|\geq\cos(5\pi/12)\|\underline{e}_i-\underline{e}_1\|\alpha_1(\proj B_{\underline{j}}|V).
	$$
	The claim of the lemma follows from the combination of the previous inequality with \eqref{eq:restrictedsingvalue}.
\end{proof}

\subsection{Upper bound}

The proof of the upper bound is standard and follows easily from \cite{Fal2}, but we give here the details for the convenience of the reader. First, we show that $s_0\leq2$. %The proof is similar to the proof of Lemma~\ref{lem:upperbound}. However, it requires the more sophisticated estimates from Section~\ref{sec:singvalsep}.

\begin{lemma}\label{lem:ub}
Under the assumptions above, $s_0\leq2$. %There exists	$\sum_{\underline{j},\underline{\sigma}}\varphi^s(B_{\underline{j},\underline{\sigma}})<\infty$ for every $s>2$.
\end{lemma}

\begin{proof}
	By Lemma~\ref{lem:shape} and Lemma~\ref{lem:projosc}, for every $n\in\mathbb{N}$, $\bigcup\limits_{\underline{j}\in\Sigma_n}\mathrm{proj}\circ G_{\underline{j}}(T_0)\subset T_0$ and $G_{\underline{j}}(T_0)^o\cap G_{\underline{i}}(T_0)^o=\emptyset$ for every $\underline{j}\neq\underline{i}\in\Sigma_n$.
	Hence, by Lemma~\ref{lem:area}
	\[
	\begin{split}
	\mathrm{Area}(\mathrm{proj}(T_0))&\geq\sum_{\underline{j}\in\Sigma_n}\mathrm{Area}\left(\mathrm{proj}(G_{\underline{j}}(T_0))\right)\\
		&\geq C\sum_{\underline{j}\in\Sigma_n}\alpha_1(B_{\underline{j}})\alpha_2(B_{\underline{j}}).
	\end{split}
	\]
	On the other hand, by Lemma~\ref{lem:adaptedchart is unique} and the regularity of the chipping rates $\bp$, we have that for every $j\in I^*$, $\delta<\lambda_j<1$, and so by Lemma~\ref{lem:contr}, there exists a constant $C>0$ such that for every $\underline{j}\in\Sigma_*$
	\begin{equation}\label{eq:Biscontr}
	\alpha_3(B_{\underline{j}})\leq\alpha_1(B_{\underline{j}})\leq C(1-\delta)^{|\underline{j}|}.
	\end{equation}
	Thus, for every $s>2$
	\[
	\begin{split}
		\sum_{n=1}^\infty\sum_{\underline{j}\in\Sigma_n}\alpha_1(B_{\underline{j}})\alpha_2(B_{\underline{j}})\alpha_3(B_{\underline{j}})^{s-2}&\leq C^{s-2}\sum_{n=1}^\infty(1-\delta)^{n(s-2)}<\infty.
	\end{split}\]
\end{proof}

\begin{lemma}\label{lem:upperbound}
	Under the assumptions above, $\overline{\dim}_B(X)\leq s_0$.
\end{lemma}

\begin{proof}
	Let $s_0<s<3$ be arbitrary but fixed, and let $K:=\sum_{\underline{j}\in\Sigma_*}\varphi^s(B_{\underline{j}})$, and let $\ell=\lceil s\rceil$. For every infinite mother sequence $\underline{j}\in\Sigma$ and $\delta>0$, there exists a unique $n=n(\underline{j},\delta)\in\mathbb{N}$ such that $\alpha_\ell(\boldsymbol{A}_{j_0}B_{\underline{j}|_n})\leq\delta<\alpha_\ell(\boldsymbol{A}_{j_0}B_{\underline{j}|_{n-1}})$. Let $$\mathcal{M}_\delta=\{\underline{j}\in\Sigma_*:\alpha_\ell(\boldsymbol{A}_{j_0}B_{\underline{j}})\leq\delta<\alpha_\ell(\boldsymbol{A}_{j_0}B_{\underline{j}|_{|\underline{j}|-1}})\}.$$
	
	Let $B(\underline{0},1)$ be the unit ball centred at the origin. Since $T\subset B(\underline{0},1)$, $X\subset\bigcup_{j_0\in I(P),\underline{j}\in\mathcal{M}_\delta}F_{j_0}\circ G_{\underline{j}}(B(\underline{0},1))$ for every $\delta>0$. Furthermore, $F_{j_0}\circ G_{\underline{j}}(B(\underline{0},1))$ is an ellipse with main semi-axis of length $\alpha_i(\boldsymbol{A}_{j_0}B_{\underline{j}})$. Let $R_{\underline{j}}$ be the smallest closed rectangle with axis parallel to the main axis of $F_{j_0}\circ G_{\underline{j}}(B(\underline{0},1))$. Then, for every $\underline{j}\in\mathcal{M}_\delta$, $R_{\underline{j}}$ can be covered by at most $\prod_{n=1}^{\ell-1}\left\lceil\frac{\alpha_n(\boldsymbol{A}_{j_0}B_{\underline{j}})}{\alpha_\ell(\boldsymbol{A}_{j_0}B_{\underline{j}})}\right\rceil$-many cubes of side length $\delta>0$. Thus,
	\[
	\begin{split}
		N_{\sqrt{3}\delta}(X)&\leq\sum_{\underline{j}\in\mathcal{M}_\delta}\prod_{n=1}^{\ell-1}\left\lceil\frac{\alpha_n(\boldsymbol{A}_{j_0}B_{\underline{j}})}{\alpha_\ell(\boldsymbol{A}_{j_0}B_{\underline{j}})}\right\rceil\leq 4\sum_{\underline{j}\in\mathcal{M}_\delta}\prod_{n=1}^{\ell-1}\frac{\alpha_n(\boldsymbol{A}_{j_0}B_{\underline{j}})}{\alpha_\ell(\boldsymbol{A}_{j_0}B_{\underline{j}})}\\
		&\leq4\delta^{-s}\sum_{\underline{j}\in\mathcal{M}_\delta}\prod_{n=1}^{\ell-1}\frac{\alpha_n(\boldsymbol{A}_{j_0}B_{\underline{j}})}{\alpha_\ell(\boldsymbol{A}_{j_0}B_{\underline{j}})}\alpha_\ell(\boldsymbol{A}_{j_0}B_{\underline{j}})^s\\
		&\leq 4\delta^{-s}\sum_{\underline{j}\in\mathcal{M}_\delta}\varphi^s(\boldsymbol{A}_{j_0}B_{\underline{j}})\leq 4\delta^{-s}\max_{j\in I(P)}\varphi^s(\boldsymbol{A}_j)K,
	\end{split}\]
	which completes the proof since $s_0<s$ was arbitrary.
\end{proof}

\subsection{Lower bound}

Before we turn into the lower bound, observe that by $G_{j',j}(W)\subset W$ we get that for every $j\in I(P)$ there exists a non-empty compact set $Z_j=\bigcap_{n=1}^\infty\bigcup_{\underline{j}\in\Sigma_n}G_{\underline{j}}(W)$, moreover, $X=\bigcup_{j\in I(P)}F_j(Z_j)$.

\begin{lemma}\label{lem:divergent}
	Under the assumptions above, $s_0\geq1$.
\end{lemma}

\begin{proof}
	Clearly, for every $i\in I(P)$, $Z_i$ contains a curve connecting $\underline{e}_1$ and $\underline{e}_2$. Let us denote this curve by $\Gamma$. Let $\mathcal{D}_n(i):=\{\underline{j}\in\Sigma_n:G_{\underline{j}}(W)^o\cap\Gamma\neq\emptyset\}$. By \eqref{eq:images}, we can order $\mathcal{D}_n(i)=\{\underline{j}_1,\ldots,\underline{j}_{\#\mathcal{D}_n(i)}\}$ such that   $G_{\underline{j}_\ell}(\underline{e}_1)=G_{\underline{j}_{\ell+1}}(\underline{e}_2)$. Hence,
	\[
		\sqrt{2}=\|\underline{e}_2-\underline{e}_1\|\leq\sum_{\underline{j}\in\mathcal{D}_n(i)}\|B_{\underline{j}}(\underline{e}_2-\underline{e}_1)\|\leq\sum_{\underline{j}\in\Sigma_n}\alpha_1(B_{\underline{j}})\|\underline{e}_2-\underline{e}_1\|,
\]
which implies $\sum_{\underline{j}\in\Sigma_*}\alpha_1(B_{\underline{j}})=\infty$.
\end{proof}

Let us now define a modified cut of the mother sequences: let $i\in I(P)$, and let
$$
\mathcal{M}_n(i):=\{\underline{j}\in \Sigma_*:j_0=i\text{ and }C(1-\delta)^{n+1}<\alpha_2(B_{\underline{j}})\leq C(1-\delta)^n\},
$$
where $C(1-\delta)^n$ is the upper estimate in Lemma~\ref{lem:contr}. Hence, for every $\underline{j}\in\mathcal{M}_n(i)$, $|\underline{j}|\leq n$. Furthermore, $\Sigma_*=\bigcup_{i\in I(p)}\bigcup_{n=1}^\infty\mathcal{M}_n(i)$.

%Observe that there exists a constant $c:=\delta^2/25>0$ such that
%\begin{equation}\label{eq:alpha2}
%\alpha_2(B_{\underline{j},\underline{\sigma}})=\frac{\|B_{\underline{j},\underline{\sigma}}^{\wedge^2}\|}{\|B_{\underline{j},\underline{\sigma}}\|}\geq\frac{\|B_{\underline{j}_-,\underline{\sigma}_-}^{\wedge^2}\|}{\|B_{\underline{j}_-,\underline{\sigma}_-}\|}\frac{\|\left(B_{j_n,\sigma_n}^{\wedge^2}\right)^{-1}\|^{-1}}{\|B_{j_n,\sigma_n}\|}\geq\frac{\delta^2}{25}\alpha_2(B_{\underline{j}_-,\underline{\sigma}_-}).
%\end{equation}

%For short, let $\mathcal{N}_n=\mathcal{M}_{C(1-\delta)^n}$,  Without loss of generality, we may choose $C$ such that $\mathcal{N}_1=I(P)$.

\begin{lemma}\label{lem:subseq}
If $\sum_{\underline{j}\in\Sigma_*}\varphi^s(B_{\underline{j}})=\infty$ then there exists a sequence $n_k$ and an $i\in I(P)$ such that $\sum_{\underline{j}\in\mathcal{M}_{n_k}(i)}\varphi^s(B_{\underline{j}})> n_k^{-2}.$
\end{lemma}

\begin{proof}
Let us argue by contradiction. Suppose that there exists $N\geq1$ such that for every $n\geq N$ and $i\in I(P)$
$$
\sum_{\underline{j}\in\mathcal{M}_{n}(i)}\varphi^s(B_{\underline{j}})\leq n^{-2}.
$$
Hence,
\[
\begin{split}
\sum_{\underline{j}\in\Sigma_*}\varphi^s(B_{\underline{j}})&=\sum_{i\in I(P)}\sum_{n=1}^\infty\sum_{\underline{j}\in\mathcal{M}_n(i)}\varphi^s(B_{\underline{j}})\\
&\leq \sum_{i\in I(P)}\sum_{n=1}^N\sum_{\underline{j}\in\mathcal{M}_n(i)}\varphi^s(B_{\underline{j}})+\#I(P)\sum_{n=N}^\infty \frac{1}{n^2}<\infty,
\end{split}
\]
which is a contradiction.
\end{proof}

\begin{lemma}\label{lem:lbdim}
	For every $s\geq1$ and $i\in I(P)$ such that $\sum\limits_{\underline{j}\in\Sigma_*:j_0=i}\varphi^s(B_{\underline{j}})=\infty$, $\overline{\dim}_B\mathrm{proj}(Z_i)\geq s$.
\end{lemma}

\begin{proof}
	Trivially, $\mathrm{proj}(Z_i)\cap\mathrm{proj} G_{\underline{j}}(T_0)$ contains curves connecting $\mathrm{proj} G_{\underline{j}}(\underline{e}_1)$ and $\mathrm{proj} G_{\underline{j}}(\underline{e}_2)$, $\mathrm{proj} G_{\underline{j}}(\underline{e}_1)$ and $\mathrm{proj} G_{\underline{j}}(\underline{e}_3)$. Then for every $\underline{j}\in\Sigma_*$
	$$
	N_{r}\left(\mathrm{proj}(Z_i)\cap\mathrm{proj}(G_{\underline{j}}(T_0))\right)\geq\frac{\|\proj B_{\underline{j}}(\underline{e}_k-\underline{e}_1)\|}{r}
	$$
	Hence, by choosing $r=C(1-\delta)^{n+1}$ and $\underline{j}\in\mathcal{M}_n$, and applying Lemma~\ref{lem:length} we get
	$$
	N_{C(1-\delta)^{n+1}}\left(\mathrm{proj}(Z_i)\cap\mathrm{proj}(G_{\underline{j}}(T_0))\right)\geq C'\frac{\alpha_1(B_{\underline{j}})}{\alpha_2(B_{\underline{j}})}
	$$
	for some uniform constant $C'>0$. Since for every $x\in\proj(Z_i)$ there exist at most $2$ $\underline{j}\in\Sigma_m$ for every $m=1,2,\ldots$ such that $x\in\proj G_{\underline{j}}(T_0)$ we get that for every $m\leq n$
	\[
	\begin{split}
	N_{C(1-\delta)^{n+1}}(\mathrm{proj}(Z_i))&\geq N_{C(1-\delta)^n}\left(\mathrm{proj}(Z_i)\cap\bigcup_{\underline{j}\in\mathcal{M}_n(i)\cap\Sigma_m}\proj G_{\underline{j}}(T_0)\right)\\
	&\geq \frac{C'}{2}\sum_{\underline{j}\in\mathcal{M}_n(i)\cap\Sigma_m}\frac{\alpha_1(B_{\underline{j}})}{\alpha_2(B_{\underline{j}})}
\end{split}
\]
and so
\[
\begin{split}
	nN_{C(1-\delta)^{n+1}}(\mathrm{proj}(Z_i))\geq\frac{C'}{2}\sum_{\underline{j}\in\mathcal{M}_n(i)}\frac{\alpha_1(B_{\underline{j}})}{\alpha_2(B_{\underline{j}})}
	\geq\frac{C'}{2C^s(1-\delta)^{sn}}\sum_{\underline{j}\in\mathcal{M}_n(i)}\varphi^s(B_{\underline{j}})
\end{split}
\]
Hence, by Lemma~\ref{lem:subseq}, there exists a subsequence $n_k$ such that
$$
N_{C(1-\delta)^{n_k+1}}(\mathrm{proj}(Z_i))\geq\frac{C'}{2C^s(1-\delta)^{sn_k}n_k^3},
$$
which implies the claim.
\end{proof}

\begin{proof}[Proof of Theorem~\ref{thm:main}]
	By Lemma~\ref{lem:upperbound} and \eqref{eq:s0isB}, the upper bound $\overline{\dim}_B(X)\leq s_0$ follows. 
	
	For the lower bound, we have by Lemma~\ref{lem:lbdim} that for every $s\geq1$ with $s<s_0$ or $s=1$
	\[
	\begin{split}
	\overline{\dim}_B(X)&=\max\{\overline{\dim}_B(F_j(Z_j)):j\in I(P)\}\\
	&\geq\max\{\overline{\dim}_B(Z_j):j\in I(P)\}\\
	&\geq\max\{\overline{\dim}_B(\proj(Z_j)):j\in I(P)\}\geq s.
\end{split}
\]
Since $s$ was arbitrary, the claim follows.
\end{proof}

{
\section{The Hausdorff dimension of the edge net of the abraded polyhedron with constant chipping rate}
\label{sec:proof2}

In this section, we turn to the case when the chipping rates are constant. As we have seen, $X=\bigcup_{F'\in\mathcal{F}^{(0)}}F'(Y)$ in this case, where $Y$ is the unique non-empty compact set satisfying \eqref{eq:constantchipping}. That is, $X$ is a finite union of affine images of the self-affine set $Y$. 

Let us now introduce two new definitions. A finite collection of $d\times d$ matrices $\{A_1,\ldots,A_m\}$ is called {\it proximal} if there exists a finite product $A_{i_1}\cdots A_{i_n}$ such that the maximal eigenvalue in modulus has multiplicity one. Moreover, we call the collection of $d\times d$ matrices $\{A_1,\ldots,A_m\}$ {\it strongly irreducible} if there is no finite collection of proper subspaces $V_1,\ldots,V_n$ of $\R^d$ such that $A_i\bigcup_{j=1}^nV_j\subseteq\bigcup_{j=1}^nV_j$ for every $i=1,\ldots,m$. 

We state the result of B\'ar\'any, Hochman and Rapaport \cite[Theorem~1.1]{BHR} and Morris and Sert \cite[Theorem~1.5]{MS}, which we intend to apply to prove Theorem~\ref{thm:main2}.

\begin{theorem}\label{thm:BHRMS}
	Let $\Phi=\{F_i(x)=A_ix+t_i\}_{i=1}^m$ be a finite collection of affine contractions on $\R^d$ for $d=2,3$. Suppose that the matrices $\{A_1,\ldots,A_m\}$ are invertible, proximal, strongly irreducible and there exists an open and bounded set $U$ such that $F_i(U)\cap F_j(U)=\emptyset$ for every $i\neq j$ and $F_i(U)\subseteq U$ for every $i$. Then
	$$
	\dim_H(X)=\inf\{s>0:\sum_{n=1}^\infty\sum_{j_1,\ldots,j_n=1}^m\varphi^s(A_{j_1}\cdots A_{j_n})<\infty\},
	$$
	where $X$ is the unique non-empty compact set such that $X=\bigcup_{i=1}^mF_i(X)$.
\end{theorem}

Our main goal is now to verify the conditions of Theorem~\ref{thm:BHRMS} for the system in \eqref{eq:constantchipping}. For the convenience of the reader, let us recall that
$$
C_1=\begin{bmatrix}	1-2p & -p & -p \\ 0 & p & 0 \\ 0 & 0 & p\end{bmatrix},\ C_2=\begin{bmatrix}	p & 0 & 0 \\ -p & 1-2p & -p \\ 0 & 0 & p\end{bmatrix},\ C_3=\begin{bmatrix}	p & 0 & 0 \\ 0 & p & 0 \\ -p & -p & 1-2p\end{bmatrix}.
$$

\begin{lemma}\label{lem:proxy}
The matrices $\{C_1,C_2,C_3\}$ are proximal for every $0<p<1/2$.
\end{lemma}

\begin{proof}
Since the eigenvalues of the matrix $C_1$ are $1-2p$ with multiplicity one and $p$ with multiplicity two, the proximality is straightforward if $p<1/3$. For the case $p\geq1/3$, consider the product
$$
C_1C_2=\begin{bmatrix}
	p-p^2 & -(1-2p)p & 0 \\ -p^2 & (1-2p)p & -p^2 \\ 0 & 0 & p^2
\end{bmatrix},
$$
which has eigenvalues $$0<\frac{p\left(2-3p-\sqrt{(4-7p)p}\right)}{2}<p^2<\frac{p\left(2-3p+\sqrt{(4-7p)p}\right)}{2}$$ for $1/3\leq p<1/2$, which completes the proof.
\end{proof}

\begin{lemma}\label{lem:irred3}
The matrices $\{C_1,C_2,C_3\}$ are strongly irreducible for every $p\in(0,1/2)\setminus\{1/5\}$.
\end{lemma}

\begin{proof}
Let us argue by contradiction. That is, suppose that there exists a finite collection of proper subspaces $V_1,\ldots,V_n$ of $\R^3$ such that $C_j\bigcup_{i=1}^nV_i\subseteq\bigcup_{i=1}^nV_i$ for every $j=1,2,3$. Since the matrices $\{C_1,C_2,C_3\}$ are invertible, the image of a two-dimensional subspace is a two-dimensional subspace. Let $\{W_1,\ldots,W_m\}\subseteq \{V_1,\ldots,V_n\}$ be the subset containing the two-dimensional subspaces. Using that the matrices are invertible, $C_j\bigcup_{i=1}^mW_i=\bigcup_{i=1}^mW_i$ for every $j=1,2,3$.

Let $W_i^\perp$ be the orthogonal complement of $W_i$. By a similar argument like in Lemma~\ref{lem:vectors}, we get that $(C_j^{-1})^T\bigcup_{i=1}^mW_i^\perp=\bigcup_{i=1}^mW_i^\perp$. In particular, for every $i=1,2,3$, there exists a permutation map $q_i\colon\{1,\ldots,m\}\to\{1,\ldots,m\}$ such that $(C_i^{-1})^TW_j=W_{q_i(j)}$. Since the permutations form a finite group, by Lagrange's theorem $q_i^{m!}$ is the identity map, i.e. $(C_i^{-m!})^TW_j=W_j$ for every $i=1,2,3$ and $j=1,\ldots,m$. That is, $W_j$ is a one-dimensional eigenspace of $(C_i^{-m!})^T$ and simple calculations show that $W_j$ must be a one-dimensional eigenspace of $(C_i^{-1})^T$ for every $j=1,\ldots,m$ and $i=1,2,3$. Hence, for $i=1$
$$
W_j^\perp=\left\{(a(1-3p), -pa, -pa):a\in\R\right\}\text{ or }W_j^\perp\subset\left\{(0, a, b):a,b\in\R\right\}.
$$
If $p\neq1/3$, then the first case is not possible since it is not invariant with respect to $(C_2^{-1})^T$. So $$W_j^\perp\subset\left\{(0, a, b):a,b\in\R\right\}\cap\left\{(a, 0, b):a,b\in\R\right\}\cap\left\{(a, b, 0):a,b\in\R\right\},$$ which is impossible. If $p=1/3$ then the subspace in the first case is contained in the subspace in the other possibility, and so, every subspace in the collection $\{V_1,\ldots,V_n\}$ must be one-dimensional.

However, similar argument shows that then every subspace $V_j$ must be a one-dimensional eigenspace of every matrix $C_i$. It is easy to see that the eigenspace of the eigenvalue $1-2p$ of any of the three matrices is not invariant with respect to the other matrices. If $p=1/3$ then there are no further eigenspaces. If $p\neq1/3$, then simple algebraic manipulations show that every vector in the two-dimensional eigenspace of $C_i$ corresponding to the eigenvalue $p$ is an eigenvector. The three two-dimensional eigenspace of the matrices $\{C_1,C_2,C_3\}$ has non-trivial intersection if and only if 
$$
0=\det\left(\begin{bmatrix}
	1 & \frac{-p}{1-3p} & \frac{-p}{1-3p} \\
	\frac{-p}{1-3p} & 1 & \frac{-p}{1-3p} \\
	\frac{-p}{1-3p} & \frac{-p}{1-3p} & 1
\end{bmatrix}\right),
$$
where the column vectors above are the normal vectors of the two-dimensional eigenspaces corresponding to the eigenvalue $p$. But the equation above holds if and only if $p=1/5$, which completes the proof. 
\end{proof}

Now, we study the case when $p=1/5$. Let us recall from Section~\ref{sec:singvalsep} that $\proj\colon\R^3\to V$ denotes the orthogonal projection to the subspace $V$ perpendicular to the vector $(1,1,1)$. If $p=1/5$ then the vector $(1,1,1)$ is an eigenvector for all of the matrices $C_1,C_2$ and $C_3$ with eigenvalue $1/5$. Thus, the linear map $D_j:=\proj C_j|_V\colon V\to V$ satisfies that
\begin{equation}\label{eq:inv}
D_j\proj(\underline{v})=\proj(C_j\underline{v})\text{ for every }\underline{v}\in\R^3.
\end{equation}
It is easy to see that $D_j$ is invertible for every $j=1,2,3$.

\begin{lemma}\label{lem:irred2}
	For $p=1/5$, the linear maps $\{D_1,D_2,D_3\}$ are strongly irreducible on $V$.
\end{lemma}

\begin{proof}
	Let us again argue by contradiction. Let $\{W_1,\ldots,W_m\}$ be one-dimensional proper subspaces of $V$ such that $D_j\bigcup_{i=1}^mW_i=\bigcup_{i=1}^mW_i$. Similar argument to Lemma~\ref{lem:irred3} shows that every $W_i$ must be an eigenspace of $D_j$ for every $j=1,2,3$. By \eqref{eq:inv}, the orthogonal projections of the eigenvectors of $C_j$ are the eigenvectors of $D_j$. Hence, for the matrix $D_1$, 
	$$
	W_j=\{(0,a,-a):a\in\R\}\text{ or }W_j=\{(2a,-a,-a):a\in\R\}
	$$ for every $j=1,\ldots,m$. However, none of these subspaces are invariant with respect to $D_2$, which is a contradiction.
\end{proof}

\begin{proof}[Proof of Theorem~\ref{thm:main2}]
	Our goal is to show that the maps in \eqref{eq:constantchipping} satisfies the conditions in Theorem~\ref{thm:BHRMS}. Since $X$ is a finite union of affine images of $Y$ defined in \eqref{eq:constantchipping}, the claim of the theorem follows. 
	
	First, let us suppose that $p\in(0,1/2)\setminus\{1/5\}$. For simplicity, let $G_i(\underline{x})=C_i\underline{x}+2p\underline{e}_i$. Let $W$ be the tetrahedron defined by the vectors $\{\underline{0},\underline{e}_1,\underline{e}_2,\underline{e}_3\}$. Then by \eqref{eq:cont}, $G_i(W^o)\subset W^o$ for every $i=1,2,3$. The claim $G_i(W^o)\cap G_j(W^o)=\emptyset$ for $i\neq j$ clearly follows by Lemma~\ref{lem:projosc}. The proximality follows by Lemma~\ref{lem:proxy} and the strong irreducibility follows by Lemma~\ref{lem:irred3}.
	
	Let us turn to the case $p=1/5$. Let $g_i\colon V\to V$ be the affine maps $g_i(\underline{x})=D_i\underline{x}+2p\proj(\underline{e}_i)$. Hence, $\proj\circ G_i=g_i\circ\proj$ by \eqref{eq:inv}. Thus, $\proj(Y)$ is the unique non-empty compact set such that $\proj(Y)=\bigcup_{i=1}^3g_i(\proj(Y))$. By Lemma~\ref{lem:projosc}, we have that $g_i(T_0^o)\subset T_0^o$ for every $i=1,2,3$ and $g_i(T_0^o)\cap g_j(T_0^o)=\emptyset$ for every $i\neq j$. The linear maps $D_1,D_2,D_3$ are clearly proximal, and are strongly irreducible by Lemma~\ref{lem:irred2}. Hence,
	\begin{equation}\label{eq:lowerhaus}
	\dim_H(Y)\geq\dim_H(\proj(Y))=\inf\left\{s>0:\sum_{n=1}^\infty\sum_{j_1,\ldots,j_n=1}^3\varphi^s(D_{j_1}\cdots D_{j_n})<\infty\right\}.
	\end{equation}
	By \eqref{eq:bounds} and \eqref{eq:restrictedsingvalue}, there exists $C>0$ such that
	$$
	\alpha_1(D_{j_1}\cdots D_{j_n})\geq C\alpha_1(C_{j_1}\cdots C_{j_n})\text{ and }
	$$
	$$
	\alpha_1(D_{j_1}\cdots D_{j_n})\alpha_2(D_{j_1}\cdots D_{j_n})\geq C\alpha_1(C_{j_1}\cdots C_{j_n})\alpha_2(C_{j_1}\cdots C_{j_n}),
	$$
	which implies that the right-hand side of \eqref{eq:lowerhaus} equals $s_0$, which completes the proof.
\end{proof}
}

\bibliographystyle{abbrv}
\bibliography{bib1}

\end{document}